\documentclass{article}
\usepackage{useful_math}
\usepackage{graphicx} 

\newcommand{\SI}{\mathcal{SI}}
\newcommand{\highdeg}{\mathsf{highdeg}}

\newcommand{\cascade}{\cI}

\newcommand{\HDestimate}{\widehat{\mathrm{HD}}}

\newcommand{\hatderivative}{\widehat{dI_\beta}}
\newcommand{\derivative}{dI}

\title{Finding Super-spreaders in Network Cascades}
\author{
Elchanan Mossel\thanks{Department of Mathematics, MIT;
\url{elmos@mit.edu}}
\and 
Anirudh Sridhar\thanks{Department of Mathematics, MIT;
\url{anisri@mit.edu}}
}
\date{\today}

\begin{document}

\maketitle

\begin{abstract}%
  Suppose that a cascade (e.g., an epidemic) spreads on an unknown graph, and only the infection times of vertices are observed. What can be learned about the graph from the infection times caused by multiple distinct cascades? Most of the literature on this topic focuses on the task of recovering the \emph{entire} graph, which requires $\Omega ( \log n)$ cascades for an $n$-vertex bounded degree graph. Here we ask a different question: can the important parts of the graph be estimated from just a few (i.e., constant number) of cascades, even as $n$ grows large?

In this work, we focus on identifying super-spreaders (i.e., high-degree vertices) from infection times caused by a Susceptible-Infected process on a graph. Our first main result shows that vertices of degree greater than $n^{3/4}$ can indeed be estimated from a constant number of cascades. Our algorithm for doing so leverages a novel connection between vertex degrees and the second derivative of the cumulative infection curve. Conversely, we show that estimating vertices of degree smaller than $n^{1/2}$ requires at least $\log(n) / \log \log (n)$ cascades. Surprisingly, this matches (up to $\log \log n$ factors) the number of cascades needed to learn the \emph{entire} graph if it is a tree.%
\end{abstract}

\section{Introduction}

Cascading behaviors in networks are ubiquitous. In human interaction networks, the spread of diseases through local interactions can quickly escalate, causing population-level pandemics \cite{brauer2012mathematical}. Similarly, the structure of social networks plays a fundamental role in the mass dissemination of ideas, such as product adoption or misinformation \cite{leskovec2007dynamics, strang1998diffusion, rogers2003diffusion, Friggeri_Adamic_Eckles_Cheng_2014}. 
In all of these examples, understanding the key factors which facilitate cascades is crucial for the analysis and control of these processes.

The structure of the underlying network plays a fundamental role in the evolution of a cascading process. Consider, for instance, the impact of a high-degree vertex. Once such a vertex is affected by the cascade, they may, in turn, disseminate the cascade's effects on a massive scale to their neighbors. Due to their potential to accelerate the course of a cascade in this manner, high-degree vertices are of special interest in the analysis of spreading processes. In the context of epidemics, quarantining or immunizing high-degree individuals, also known as \emph{super-spreaders}, can greatly mitigate the damage caused by a pandemic \cite{pastor2002immunization, dezso2002halting}. In social networks, firms may want to {leverage} the impact of high-degree vertices, also known as \emph{influencers}, to accelerate the speed of product adoption \cite{hinz2011seeding}. 

A fundamental hurdle in carrying out these ideas in practice is that the underlying network may be partially or fully unknown. In epidemiology, disease-spreading contacts between individuals are often not measured.
In marketing, even if retailers have access to certain social networks, only some of the observed links may be effective in facilitating a cascade, and many links may also be unobserved (e.g., if two individuals communicate through other means).
While it is possible to learn aspects of the network through contact tracing or interviewing, such approaches are quite intensive and time-consuming \cite{spencer2020covid}. To mitigate this issue, we aim to understand the extent to which more efficient \emph{data-driven} approaches can be used to learn about the network. Specifically, we study how network structure can be estimated from the set of ``infection times'' caused by a cascade, which we call the \emph{cascade trace}. Typically, a single cascade trace may not be enough to learn meaningful information, but much more can be said by combining information from multiple different cascades on the same network.

The task of estimating a network from multiple cascade traces, known as structure learning, has received considerable attention in the past decade \cite{gomez2012inferring, du2012learning, NS12_cascades, ACFKP13_trace_complexity, daneshmand2014estimating, khim2018theory, hoffman2019learning, he2020network, hoffman2020learning, wilinski2021prediction}. Almost all the literature on this topic focuses on exactly recovering the \emph{entire} network under sparsity constraints (i.e., bounds on the maximum degree). However, the number of cascade traces required to do so grows as a function of the network size \cite{NS12_cascades, ACFKP13_trace_complexity, daneshmand2014estimating}. This can be prohibitive in applications such as epidemiology where it is costly to allow many cascades to spread. Motivated by this limitation, we ask a different question:
\begin{center}
\label{q:high_degree}
\emph{Can the \underline{important} parts of the network, such as the location of high-degree vertices, \\
be learned from just a few cascade traces? }
\end{center}
We find that the answer is quite subtle. In an $n$-vertex graph, we show that vertices of degree at least $n^{3/4}$ can indeed be correctly estimated from a constant number of cascade traces. Conversely, estimating vertices of degree smaller than $\sqrt{n}$ can be almost as hard as learning the \emph{entire} network, in terms of the number of cascade traces required.

\subsection{Summary of contributions}

In a bit more detail, we assume that each cascade spreads on an $n$-vertex graph $G$ with the following structure: a constant number of (high-degree) vertices have degree at least $n^\alpha$ for some $\alpha \in (0,1)$, and all other vertices have degree at most $n^{o(1)}$. The goal is to understand how many cascade traces are needed to estimate the set of high-degree vertices with probability $1 - o(1)$.

\begin{figure}
    \centering
    \includegraphics[width=0.8\textwidth]{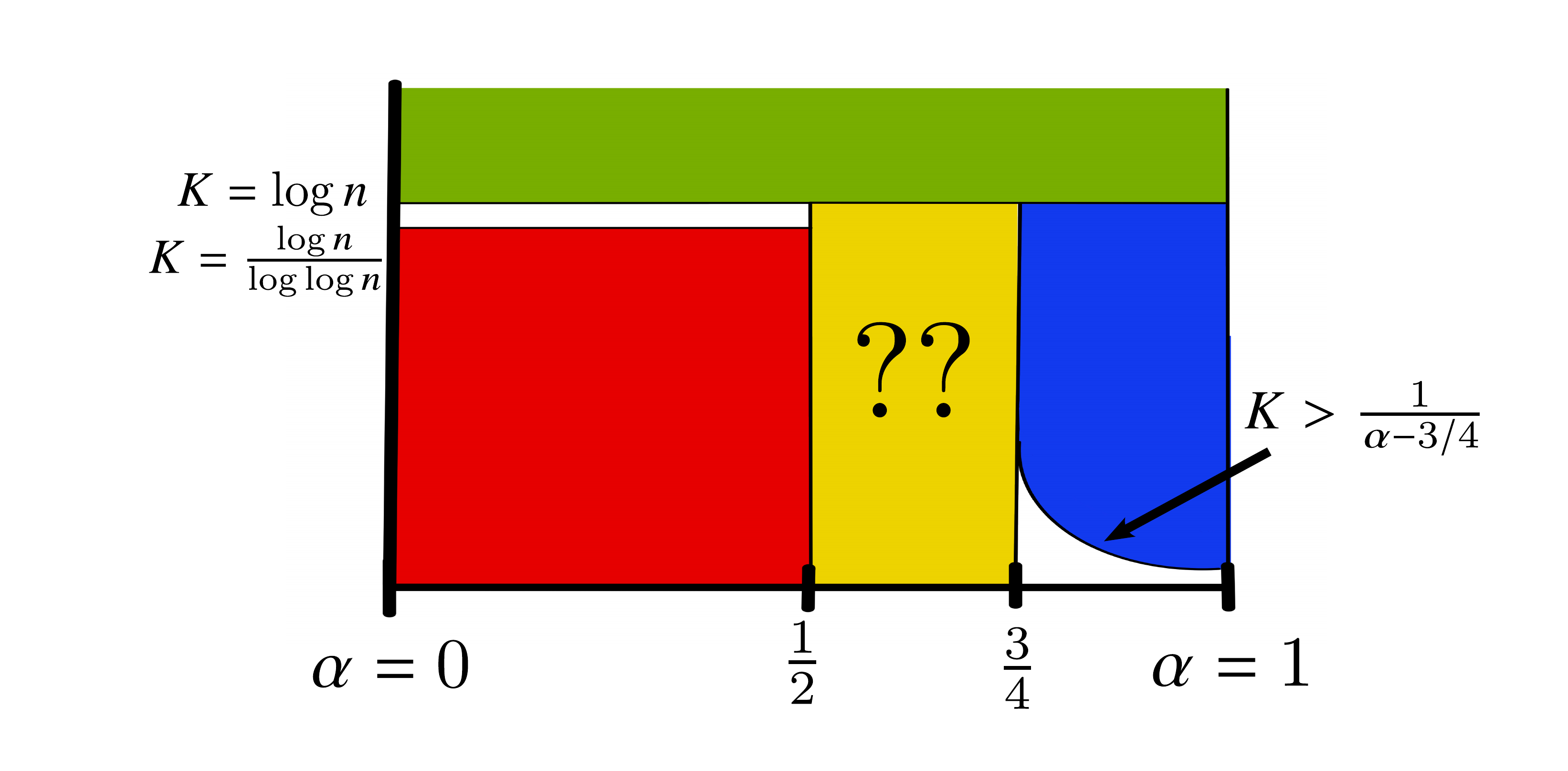}
    \caption{Phase diagram for the possibility and impossibility of estimating high-degree vertices in a graph $G$.
    \emph{Blue region:} High-degree vertices can be estimated from more than $1/(\alpha - 3/4)$ traces.
    \emph{Red region:} Estimating high-degree vertices is impossible, even when $G$ is known to be a tree.
    \emph{Green region:} Full recovery of $G$ is possible if $G$ is a tree, hence estimation of high-degree vertices is also possible.
    \emph{White region:} Regimes with small gap between the bounds provided by our analysis. 
    \emph{Yellow region:} The main open problem left given our work is to determine the sample complexity in this region. 
    }
    \label{fig:phase_diagram}
\end{figure}

Our first main contribution is an estimator for the set of high-degree vertices which requires only a constant number of cascade traces. Our estimator is quite simple, and relies only on properties of the \emph{infection curve} $I(t)$, which counts the number of infections that occur before time $t$. The key idea behind our method is a novel connection between vertex degrees and the second derivative of the infection curve. Specifically, we show that if $v$ becomes infected at time $T(v)$, then a discretized version of the second derivative of $I(t)$ evaluated at time $T(v)$ is a (nearly) unbiased estimator for the degree of $v$. Motivated by this insight, our estimator computes the (appropriately discretized) second derivative of $I(t)$ at $T(v)$ for each $v$, and compares it to a threshold. The set of estimated high-degree vertices is then given by all vertices which pass the threshold in all of the observed cascade traces (see Algorithm \ref{alg:second_derivative} for more details). Our analysis shows that this estimator succeeds provided $\alpha > 3/4$. Notably, our algorithm is a significant departure from existing methods in the literature on learning in graphical models. Indeed, prior algorithms typically focus on learning the precise edges and neighborhoods in the graph \cite{NS12_cascades, ACFKP13_trace_complexity, daneshmand2014estimating, gomez2012inferring, du2012learning, he2020network, wilinski2021prediction}. In contrast, we demonstrate how degree information can be \emph{directly} estimated from just a few traces without needing to infer neighborhoods. 

A natural follow-up question is whether one can identify high-degree vertices of degree much smaller than $n^{3/4}$, while still using only a constant number of traces. We provide a negative answer, showing that when $\alpha \in (0,1/2)$, there exist ``hard'' instances for which $\log (n) / \log \log (n)$ traces are needed to estimate high-degree vertices, no matter the estimator. In other words, it is \emph{information-theoretically impossible} to correctly estimate high-degree vertices in this case. In our proof, we show that estimating high-degree vertices in a particular ensemble of graphs is equivalent to a variant of the \emph{sparse mixture detection problem} \cite{dobrushin2958statistical, burnashev1991problem, ingster1996some, donoho2004higher, cai2014optimal}. Following the approach of Cai and Wu \cite{cai2014optimal}, we solve our version of the sparse mixture detection problem, which leads to a lower bound of $\log (n) / \log \log (n)$ traces for the estimation of high-degree vertices. 
A surprising implication of our impossibility result is that for $\alpha \in (0,1/2)$, estimating high-degree vertices can be nearly as hard as learning the entire network. Indeed, using $\log(n)$ traces, the Tree Reconstruction algorithm of Abrahao, Chierichetti, Kleinberg and Panconesi \cite{ACFKP13_trace_complexity} is able to exactly recover $G$, provided it is a tree. On the other hand, the graphs in our ensemble of hard instances are all trees, hence $\log(n) / \log \log (n)$ traces are required to estimate high-degree vertices even if $G$ is known to be a tree. In summary, if $G$ is known to be a tree, the sample complexity (with respect to the number of cascade traces) of estimating high-degree vertices matches that of learning $G$ \emph{exactly}, up to at most a $\log \log (n)$ factor.  

Together, our two results highlight a curious phase transition in $\alpha$ (see Figure \ref{fig:phase_diagram}). When $\alpha \in (3/4, 1)$, high-degree vertices can be identified using essentially minimal data about the cascade. When $\alpha \in (0,1/2)$, we have the other extreme: learning high-degree vertices is nearly as challenging as learning the entire network, when $G$ is a tree. Our work leaves open the intriguing regime of $\alpha \in [1/2, 3/4]$.

\subsection{Related Work}

The problem of inferring a network from cascade traces was first studied in the context of epidemiology \cite{WT04_epidemic} and information flow in blogs \cite{AA05_blog}. Since then, there has been a large body of theoretical and empirical work on the subject. On the empirical side, Gomez-Rodriguez, Leskovec and Krause \cite{gomez2012inferring} initiated a study of scalable and mathematically principled methods with their NetInf algorithm. Since then, a number of follow-up works have tackled the exact network inference problem using a combination of likelihood-based approaches, convex optimization, message passing and machine learning \cite{daneshmand2014estimating, gomez2012inferring, du2012learning, he2020network, wilinski2021prediction}. In a related vein, there has been significant interest from the machine learning community in the task of learning Hawkes processes from vertex activation times; see \cite{lima2023hawkes} and references therein. We expect that our methods may be useful in this setting as well.

The first results on the sample complexity (i.e., number of traces needed) of exactly recovering the underlying graph were by Netrapalli and Sanghavi \cite{NS12_cascades}, which applied to discrete-time cascades (i.e., the independent cascade model) spreading on bounded-degree graphs. This work was soon followed by Abrahao, Chierichetti, Kleinberg and Panconesi \cite{ACFKP13_trace_complexity}, who studied the sample complexity of exact network inference using a variant of the SI process to model the cascade. Since their work, a number of extensions have been explored, including inference from noisy infection times \cite{hoffman2019learning}, correlated cascades \cite{khim2018theory} and learning mixtures of graphs \cite{hoffman2020learning}. A common technique underlying these is an analysis of maximum likelihood estimator for the underlying network. We note that our algorithm is a significant departure from these methods, as it is able to directly estimate degree information \emph{without} knowledge of the precise topology.

Our work falls under the broader umbrella of literature studying the sample complexity of learning structures in graphical models. Our work deals with observations from \emph{transient} models, though learning from \emph{stationary} models such as Gaussian Graphical Models and Ising models have also received considerable attention 
\cite{ravikumar2011high, meinshausen2006high, misra2020information,wang2010information,cai2016estimating,kelner2020learning, klivans2017learning, wu2019sparse, vuffray2020efficient, santhanam2012information, bresler2015efficiently, bresler2013reconstruction}. While such frameworks are naturally quite different from ours, there are also many qualitative similarities. 
For the task of recovering the precision matrix of a $n$-dimensional Gaussian graphical model, it is known that $\Theta(\log n)$ independent samples are necessary and sufficient under some sparsity conditions \cite{misra2020information,wang2010information,cai2016estimating,kelner2020learning}. Similarly, $\Theta( \log n)$ independent samples are necessary and sufficient for learning bounded-width Ising models \cite{BrMoSl:08, wu2019sparse, vuffray2020efficient, santhanam2012information, bresler2015efficiently} as well as general classes of Markov Random Fields \cite{bresler2013reconstruction}. 
These results parallel the sample complexity of exactly recovering the graph from cascade traces, and many algorithmic approaches are similar in flavor. 

While the work above largely focuses on recovering the underlying network, we mention a few recent papers which also aim to learn the ``important'' parts of the relevant graphical model, though with significantly different interpretations of importance.
Boix-Adser\`a, Bresler and Koehler \cite{boixadsera2021chow} as well as Bresler and Karzand \cite{bresler2020learning} focus on the task of learning tree-structured Ising models with a good prediction accuracy with respect to the ground-truth model.
Eckles, Esfandiari, Mossel and Rahimian \cite{eckles2022seeding} leverage cascade traces from a discrete-time cascade model to perform influence maximization. These works, in conjunction with ours, demonstrate the potential to learn key information about distributions and systems \emph{without} fully learning the underlying model.

\subsection{Definitions and Main Results}
\label{sec:results}

\paragraph{Notation.}
For a set $S$, we let $|S|$ denote its cardinality.
For a graph $G = (V,E)$, we denote the number of vertices by $n : = |V|$ and let $\cN(v)$ denote the neighborhood of a vertex $v$. For a subset of vertices $S \subset V$, we denote $\cut(S) : = |\{ (u,v) \in E : u \in S, v \notin S \} |$. Throughout the paper, we use standard asymptotic notation and all limits are as $n \to \infty$ unless otherwise specified.

\paragraph{The Susceptible-Infected process.}
We model the cascade using the well-known Susceptible-Infected (SI) process, which is a continuous-time Markov process in which susceptible vertices in a graph become infected upon interacting with infected neighbors. 
In more detail, we let $G$ be the graph upon which the cascade spreads.
We denote by $\cascade(t)$ the set of vertices which are infected at a time index $t \ge 0$, and we let $\lambda > 0$ represent the \emph{interaction rate} between individuals. Initially, we assume that only a single vertex $v_0$ is infected; that is, $\cascade(0) = \{ v_0 \}$. 
Given $\cascade(t)$, the cascade evolves as follows: for $v \in V \setminus \cascade(t)$, it holds that as $\epsilon \to 0$,
\begin{equation}
\label{eq:SI}
\p \left( \left. v \in \cascade(t + \epsilon) \right \vert \cF_t \right) = \epsilon \lambda | \cN(v) \cap \cascade(t) | + o(\epsilon),
\end{equation}
where $\{ \cF_t \}_{t \ge 0}$ is the natural filtration corresponding to the stochastic evolution of the cascade and $o(\epsilon) \to 0$ faster than $\epsilon \to 0$.
In words, the rate at which a susceptible vertex becomes infected is proportional to the number of its infected neighbors and the interaction rate $\lambda$. Given the evolution of the cascade $\cascade(t)$, a few important quantities include the \emph{infection curve} $I(t) : = | \cascade (t) |$, the set $\cascade[a,b]$ for $0 \le a \le b$, which denotes the set of infections occurring in the interval $[a,b]$, and $I[a,b] : = | \cascade[a,b]|$.

We remark that the SI process \eqref{eq:SI} is closely related to a number of other well-studied models. It is known to be equivalent to First Passage Percolation with $\mathrm{Exp}(\lambda)$ edge weights, a popular model in mathematical physics \cite{fpp, fpp_si_equivalence}. It is also equivalent to the Contact Process without recovery \cite{contact_process, markovian_contact_processes}, and has been used as a rigorous foundation for standard population-based epidemiological models \cite[Chapter 9]{brauer2012mathematical}. 
In recent years, a large body of work has also investigated generalizations of \eqref{eq:SI} with edge-dependent interaction rates (e.g., due to mask-wearing tendencies, multiple viral strains); see, e.g., \cite{allard2009heterogenous, alexander2010risk, eletreby2020effects}.

For a given vertex $v \in V$, we formally define its \emph{infection time} in the SI process to be
\[
T(v) : = \inf \{ t \ge 0 : v \in \cascade(t) \}.
\]
We call the collection of infection times and associated vertices $\mathbf{T} := \{(v, T(v) )\}_{v \in V}$ the \emph{cascade trace}.

\paragraph{Learning from cascade traces.} 
Suppose that the underlying graph $G$ is unknown, but we observe multiple independent cascade traces on $G$, given formally by
\begin{equation}
\label{eq:trace_distributions}
\left( \mathbf{T}_1, \ldots, \mathbf{T}_K \right) \sim \SI(G, v_{0,1}) \otimes \ldots \otimes \SI(G, v_{0,K}),
\end{equation}
where $v_{0,i}$ is the source vertex for the $i$th cascade. For brevity, for a graph $G$ and collection of source vertices $v_0 : = ( v_{0,1}, \ldots, v_{0,K})$, we denote $\p_{G, v_0}$ to be the probability measure corresponding to \eqref{eq:trace_distributions}.
Broadly, we are interested in the following question of structural inference: \emph{What can be learned about $G$ from $K$ cascade traces?}

While this has been a question of significant interest in the past decade, the existing literature on the topic almost exclusively focuses on \emph{exactly} recovering $G$. In the context of cascade traces, this task was first studied by Netrapalli and Sanghavi \cite{NS12_cascades} for a discrete-time cascade model. In particular, they proved that for $n$-vertex graphs $G$ with maximum degree at most $\Delta$, $K = \poly(\Delta) \log n$ is necessary and sufficient to exactly recover $G$. Abrahao, Chierichetti, Kleinberg and Panconesi \cite{ACFKP13_trace_complexity} studied the exact recovery problem for a slightly more general version of the SI process \eqref{eq:SI}, showing that $K = \poly(\Delta) \log(n)$ suffices for general graphs (provided $\Delta$ does not scale as a function of $n$), and that $K = \Theta ( \log n)$ traces suffices if $G$ is a tree. 

A key takeaway from this literature on exact recovery of $G$ is that the number of cascade traces $K$ must grow as a function of the network size $n$. However, this requirement may be prohibitive in settings such as epidemiology where one must make decisions based on a relatively small amount of data. In light of this issue, we study whether it is possible to learn the \emph{important} parts of a network -- such as the location of high-degree vertices -- from just a few cascade traces. To the best of our knowledge, this angle has not been previously investigated.

\paragraph{Estimating high-degree vertices.} To formalize the task of estimating high-degree vertices, we assume that most of the vertices in $G$ have a relatively small degree, except for a constant number of vertices with substantially larger degree. The latter type is the set of high-degree vertices we wish to estimate. We mathematically capture this structural property with the class of graphs $\cG(n,m,d,D)$, described below. 

\begin{definition}
\label{def:G}
We say that $G \in \cG(n,m,d,D)$ if and only if (1) $G$ is a connected graph on $n$ vertices, (2) there are at most $m$ vertices of degree at least $D$, and (3) all other vertices have degree at most $d$. We denote the set of high-degree vertices in $G \in \cG(n,m,d,D)$ to be $\highdeg(G) : = \{ v \in V : \deg(v) \ge D \}$.
\end{definition}

We remark that the assumption of connectivity ensures that all vertices are eventually infected, which is in general necessary for the estimation of high-degree vertices. Indeed, if a high-degree vertex never becomes infected, it is impossible to estimate it.
Next, we assume the following about the scaling of the parameters $m, d,D$ with respect to $n$.

\begin{assumption}
\label{as:G_conditions}
We assume that $m$ is constant, $d = n^{o(1)}$ and $D = n^\alpha$ for some $\alpha \in (0,1)$.
\end{assumption}

We remark that this assumption is quite different from the usual assumption made in the literature on structure learning. Indeed, in much of the work on exact recovery of $G$, it is assumed that its maximum degree is bounded as $n \to \infty$ \cite{NS12_cascades, ACFKP13_trace_complexity, daneshmand2014estimating}. A notable exception is the Tree Reconstruction algorithm of Abrahao, Chierichetti, Kleinberg, and Panconesi \cite{ACFKP13_trace_complexity}, which has provable guarantees for \emph{any} tree with a sufficiently large number of vertices. However, as we are only interested in estimating high-degree vertices rather than specific connections in the network, we may take significantly more relaxed assumptions compared to the bulk of the structure learning literature. In the remainder of the paper, when referencing any graph $G$ we will assume that $G \in \cG(n,m,d,D)$ and that Assumption \ref{as:G_conditions} holds.

Our first main result shows that we can correctly identify all vertices in $G$ of degree larger than $n^{3/4}$, using a \emph{constant} number of cascade traces -- far less than what is needed to learn edges in the graph.

\begin{theorem}
\label{thm:3/4}
Suppose that $\alpha \in (3/4, 1)$ and that $D = n^\alpha$. If $K$ satisfies
\[
K > \frac{1}{\alpha - 3/4},
\]
then there is an estimator $\HDestimate$ such that for any $G \in \cG(n,m,d,D)$ and any collection of source vertices $v_0 = (v_{0,1}, \ldots, v_{0,K})$,
\[
 \p_{G, v_0} \left( \HDestimate \left( \mathbf{T}_1, \ldots, \mathbf{T}_K \right) = \highdeg(G) \right ) = 1 - o(1),
\]
where $o(1) \to 0$ as $n \to \infty$.
\end{theorem}

Though Theorem \ref{thm:3/4} shows that a constant number of traces suffices for any $\alpha > 3/4$, $K$ grows as $\alpha$ approaches $3/4$. While this condition is needed in our analysis to handle worst-case examples of $G$, we remark that in practice $K$ could be smaller and may not explode as $\alpha$ approaches $3/4$.
To prove the theorem, we develop a novel algorithm for estimating high-degree vertices from a constant number of cascade traces (Algorithm \ref{alg:second_derivative}). A key supporting result underlying the theorem and Algorithm \ref{alg:second_derivative} is that the time at which a high-degree vertex is infected can be accurately estimated from just a \emph{single} cascade. See Section \ref{sec:possibility} for more details.

Our second main contribution is a negative result, showing that \emph{any} estimator requires at least $\log(n) / \log \log (n)$ traces to successfully identify high-degree vertices of degree smaller than $\sqrt{n}$.

\begin{theorem}
\label{thm:impossibility}
Fix $\epsilon > 0$, let $\alpha \in (0,1/2)$. Suppose that $D = n^\alpha$, $d \ge \log^2 n$ and $m = 1$. Additionally, assume that 
\begin{equation}
\label{eq:K_impossible}
K \le \left( \frac{1 - 2 \alpha - \epsilon}{5} \right) \frac{\log n}{\log \log n}.
\end{equation}
Then there exists an ensemble $\cH \subset \cG(n,m,d,D)$ and a probability distribution $\mu$ over graphs in $\cH$ and source vertices $v_0 = ( v_{0,1}, \ldots, v_{0,K})$ such that, for any estimator $\mathrm{HD}$,
\[
\p_{G,v_0 \sim \mu} \left( {\mathrm{HD}}(\mathbf{T}_1, \ldots, \mathbf{T}_K) = \highdeg(G) \right) = o(1).
\]
\end{theorem}

An important implication of the theorem is that if $K$ is smaller than $\log(n) / \log \log (n)$, no estimator can successfully identify high-degree vertices for every graph in the class $\cG(n,m,d,D)$. 
The ensemble $\cH$ used in the theorem is a collection of trees with a particular structure; see Sections \ref{sec:impossibility} and \ref{sec:impossibility_proof} for more details. This observation, in the context of prior work on the sample complexity of exact recovery of $G$ \cite{ACFKP13_trace_complexity}, shows that learning high-degree vertices when $\alpha \in (0,1/2)$ is \emph{nearly} as hard as learning the full graph in the case of trees. To emphasize this point, we rephrase the result of \cite{ACFKP13_trace_complexity} below.

\begin{theorem}
\label{thm:exact_recovery}
Suppose that $K \ge C \log n$, for some universal constant $C > 0$. Then there exists an estimator $\widehat{G}$ such that for any collection of source vertices $v_0 = (v_{0,1}, \ldots, v_{0,K})$ and any tree $G$, it holds that $\p_{G, v_0} ( \widehat{G} = G) = 1 - o(1)$.
\end{theorem}

Together, Theorems \ref{thm:impossibility} and \ref{thm:exact_recovery} show that there is at most a $\log \log n$ multiplicative gap between the sample complexity of estimating high-degree vertices for $\alpha \in (0,1/2)$ and that of fully recovering $G$, in the case where $G$ is a tree.

\subsection{Discussion and future work}
\label{sec:conclusion}

In this work, we study the sample complexity of estimating super-spreaders (i.e., high-degree vertices) in network cascades, using only the infection times of vertices across multiple cascades. This is the first work to consider this task, to the best of our knowledge.
Our results show that estimating vertices with degree at least $n^{3/4}$ can be done using a constant number of cascades, while estimating vertices with degree at most $n^{1/2}$ requires almost as many cascades as learning the entire graph (if the graph is a tree). We discuss several important avenues for future work below.

\begin{itemize}
    \item {\bf The regime $\alpha \in [1/2, 3/4]$.} In terms of sample complexity, our results show that the regime $\alpha \in (3/4, 1)$ is ``easy'', i.e., the number of traces needed is minimal. On the other hand, $\alpha \in (0,1/2)$ captures a ``hard'' regime, where the sample complexity nearly matches that of learning the entire graph exactly. What is the right sample complexity for $\alpha \in [1/2, 3/4]$? Is it in the easy regime, the hard regime, or something in between?

    \item {\bf Estimation in general graphs when $\alpha \in (0,1/2)$.} The Tree Reconstruction algorithm of \cite{ACFKP13_trace_complexity} provides a means of estimating high-degree vertices when the graph is a tree. What about estimating high-degree vertices for general graphs in $\cG(n,m,d,D)$?

    \item {\bf Noisy observations.} In practice, the infection times may be noisy due to delays in reporting or due to quantization (e.g., positive cases being consolidated on a day-by-day basis). To what extent can our methods be adapted to handle these more realistic scenarios?

    \item {\bf Other motifs.} A high-degree vertex can be viewed as a particular type of motif. What is the sample complexity of detecting and estimating other types of motifs (e.g., cliques)?

    \item {\bf Other types of graphical models.} Our methods fundamentally leverage the temporal nature of the observed data. How can high-degree vertices be estimated in stationary graphical models, such as Gaussian graphical models or Ising models?
    
\end{itemize}

 \section{Technical overview}
 \label{sec:overview}

 In this section, we provide a detailed overview of our methods and proof techniques. Section \ref{sec:possibility} motivates our algorithm for estimating high-degree vertices in the regime $\alpha \in (3/4, 1)$ (Algorithm \ref{alg:second_derivative}) and provides a proof sketch of Theorem \ref{thm:3/4}. In Section \ref{sec:impossibility}, we discuss the key ideas behind our impossibility result (Theorem \ref{thm:impossibility}). In particular, we elaborate on the connection between the sparse mixture detection problem and the impossibility of estimating high-degree vertices when $\alpha \in (0, 1/2)$.

 \subsection{Estimating high-degree vertices when $\alpha > 3/4$}
\label{sec:possibility}

The starting point for the design of our estimator is the following simple observation: when a high-degree vertex becomes infected, many of its neighbors will become infected shortly after. In a bit more detail, we could assess whether a vertex $v$ is high-degree by examining the number of infections occurring in a $\delta$-size interval after $v$ is infected, given formally by $I [ T(v), T(v) + \delta ]$. Since we can write 
$I [ T(v) , T(v) + \delta ] = I ( T(v) + \delta) - I (T(v))$,
the quantity $I [ T(v), T(v) + \delta]$ can be thought of as the discrete first derivative of the infection curve $I(t)$. This suggests that high-degree vertices could be estimated by examining the points at which the first derivative of $I(t)$ is large.
Unfortunately, such an approach fails in general. 
Indeed, if we observe many infections in the interval $[T(v), T(v) + \delta]$, then it is challenging to discern whether the infections were mainly caused by a \emph{single}, common infection (namely, $v$), or by many other infected vertices which just happened to be infected around the same time as $v$. 
However, it turns out that if one considers \emph{higher-order} information -- specifically, the \emph{second} derivative of $I(t)$ -- then high-degree vertices can be clearly identified.

\begin{figure}[t]
    \centering
    \begin{subfigure}{0.3 \textwidth}
        \centering
        \includegraphics[width=\textwidth]{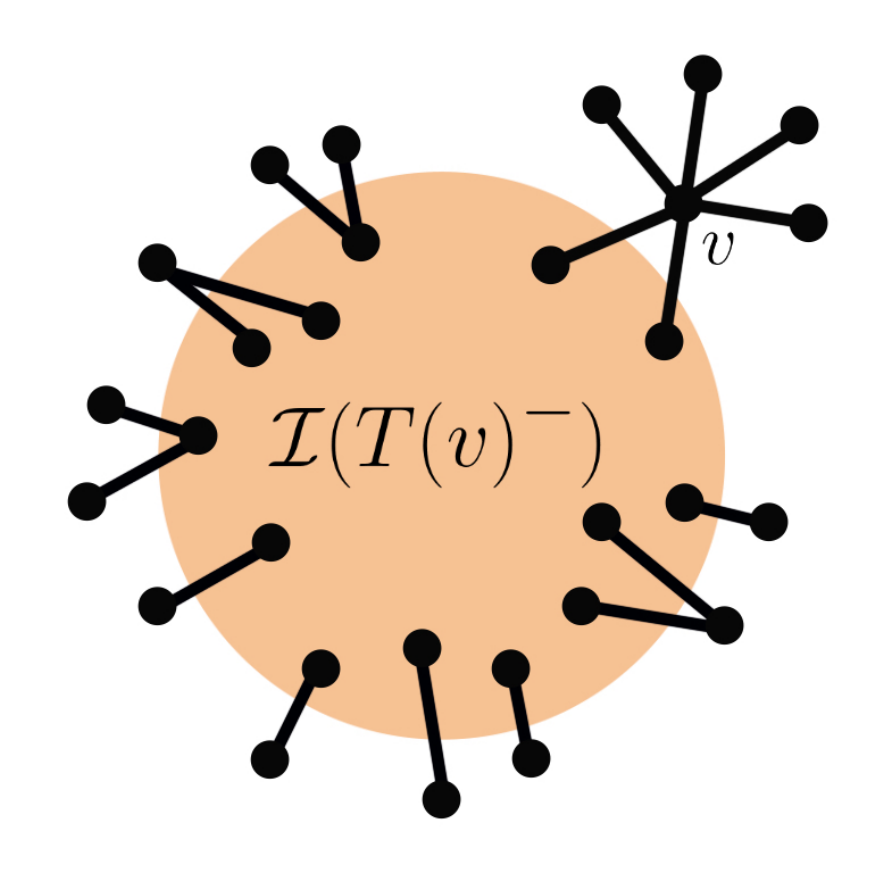}
        \caption{ }
        \label{fig:cut_before}
    \end{subfigure}%
    \begin{subfigure}{0.3 \textwidth}
        \centering
        \includegraphics[width=\textwidth]{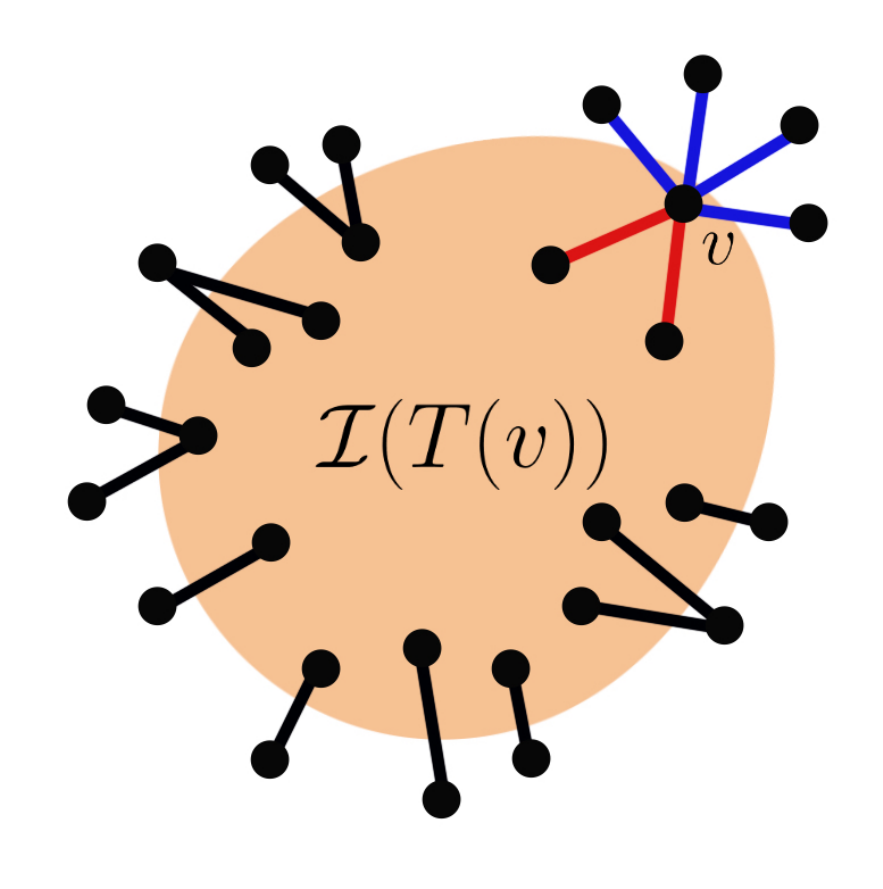}
        \caption{ }
        \label{fig:cut_after}
    \end{subfigure}
    \caption{Visualization of the edges contributing to $\cut( \cascade(t))$ before $(a)$ and after $(b)$ a vertex $v$ becomes infected. In Figure \ref{fig:cut_after}, the number of blue and red edges denote the positive and negative change to $\cut(\cascade(t))$, respectively, upon $v$ being infected.}
    \label{fig:cut}
\end{figure}

 To see this concretely, we begin by examining the first derivative of $I(t)$ in more detail. By \eqref{eq:SI},
\begin{equation}
\label{eq:J_definition}
J(t) : = \E \left[ \left. \frac{d}{dt} I(t) \right \vert \cF_t \right] = \lambda \sum_{v \in V \setminus \cascade(t)} | \cN(v) \cap \cascade(t) | = \lambda \cut ( \cascade(t)).
\end{equation}
Now suppose that a vertex $v$ joins the cascade at time $t$ (that is, $T(v) = t$). Since no other vertices can be infected at the exact same time due to the continuous-time nature of the cascade, we have
\begin{equation}
\label{eq:J}
J(t) - J(t^-)  = \lambda \left( |\cN(v) \setminus \cascade(t) | - | \cN(v) \cap \cascade(t) | \right) = \lambda \deg(v) - 2 \lambda | \cN(v) \cap \cascade(t) |.
\end{equation}
The first equality in \eqref{eq:J} can be explained as follows. When $v$ joins the cascade, all of its susceptible neighbors, i.e., those in $\cN(v) \setminus \cascade(t)$, become exposed. On the other hand, the edges which connect $v$ to $\cascade(t)$ -- namely, edges which connect $v$ to $\cN(v) \cap \cascade(t)$ -- can no longer lead to infection events, hence the contribution for these edges is subtracted. See Figure \ref{fig:cut} for a visualization of this idea.

Crucially, \eqref{eq:J} shows that as long as a vertex is infected before most of its neighborhood (i.e., $|\cN(v) \cap \cascade(t) |$ is much smaller than $\deg(v)$), the second-order changes in the infection curve, captured by $J(t) - J(t^-)$, can reveal the degree of a vertex. However, since $J(t)$ is not a directly observable quantity, we must approximate it from the cascade traces. A natural way to do so is by choosing $\delta > 0$ and using the discrete derivative $I[t, t + \delta ]/\delta$ as a proxy for $J(t)$.
Similarly, we can use $I[t- \delta, t] / \delta$ as a proxy for $J(t^-)$. We can therefore approximate the degree of $v$ using
\begin{equation*}
\widehat{\deg}_\delta (v) : = \frac{ I[ T(v), T(v) + \delta] - I [ T(v) - \delta, T(v) ]}{\delta}.
\end{equation*}
In words, we estimate the degree of $v$ by computing the discrete second derivative of the infection curve when $v$ gets infected. 
As shown through the empirical example in Figure \ref{fig:simulations}, this statistic is quite effective in identifying the infection time of a high-degree vertex from just a single cascade (this is also shown formally in Section \ref{sec:alg_proof}).

\begin{figure}[t]
    \centering
    \begin{subfigure}{0.5 \textwidth}
        \centering
        \includegraphics[width=\textwidth]{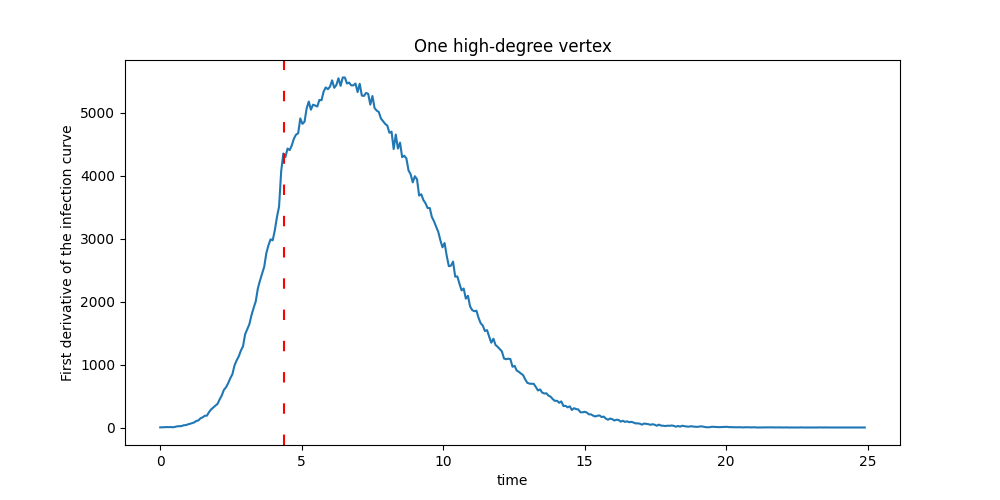}
        \caption{ }
        \label{fig:far1}
    \end{subfigure}%
    \begin{subfigure}{0.5 \textwidth}
        \centering
        \includegraphics[width=\textwidth]{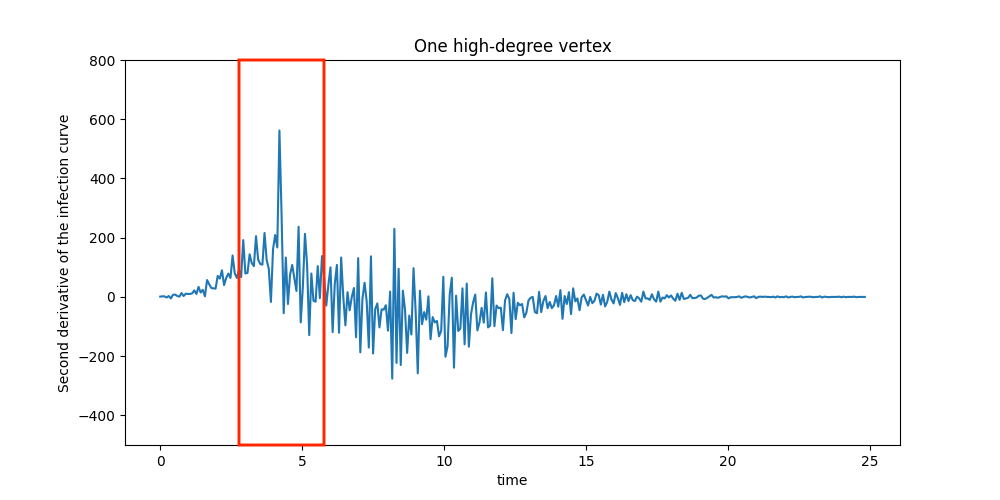}
        \caption{ }
        \label{fig:far2}
    \end{subfigure}
    \caption{Plots of the discrete first $(a)$ and second derivative $(b)$ of the infection curve with $\delta = 0.075$ generated from a graph $G$ with approximately 500,000 vertices and one high-degree vertex. The infection time of the high-degree vertex can be identified by the red dotted line in $(a)$ and the large peak in the second derivative plot $(b)$, highlighted by the red rectangle. $G$ was chosen to be a balanced, 5-regular tree of height 8, where one vertex in the 6th layer of the tree has degree 7500.}
    \label{fig:simulations}
\end{figure}

In light of \eqref{eq:J}, our estimator for $\highdeg(G)$ returns the set of all vertices $v$ for which $\widehat{\deg}_\delta(v)$ is large in all cascade traces; see Algorithm \ref{alg:second_derivative}. As Theorem \ref{thm:alg_3/4} shows, this method succeeds with high probability provided $\alpha > 3/4$.

\begin{breakablealgorithm}
\caption{Finding high-degree vertices via second derivative thresholding}
\label{alg:second_derivative}
\begin{algorithmic}[1]
\Require{Cascade traces $\mathbf{T}_1, \ldots, \mathbf{T}_K$ and parameters $\delta, \tau > 0$}
\Ensure{A set $\HDestimate \subset V$}
\State For each $i \in [K]$ and $v \in V$ compute
\[
\widehat{\deg}_{\delta, i}(v) : = \frac{ I_i [ T_i(v), T_i(v) + \delta ] - I_i [ T_i(v) - \delta, T_i(v) ]}{\delta}.
\]
\State Return $\HDestimate : = \left \{ v \in V : \widehat{\deg}_{\delta, i}(v) \ge \tau, \text{ for all $i \in [K]$} \right \}$.
\end{algorithmic}
\end{breakablealgorithm}

\begin{theorem}
\label{thm:alg_3/4}
Suppose that $\alpha > 3/4$ and that
\[
\delta = \frac{1}{n^{\alpha - 1/4}}, \hspace{1cm} \tau = \frac{n^\alpha}{\log n}, \hspace{1cm} K > \frac{1}{\alpha - 3/4}.
\]
Additionally assume that $d = n^{o(1)}$ and $D = n^\alpha$. Then for any $G \in \cG(n,m,d,D)$, the output of Algorithm \ref{alg:second_derivative} is equal to $\highdeg(G)$, with probability $1 - o(1)$.
\end{theorem}

\begin{remark}
A useful feature of Algorithm \ref{alg:second_derivative} is that it is adaptive to the unknown parameters $m$ (the number of high-degree vertices) and $\lambda$ (the spreading rate of the cascade). In contrast, the prior work on structure learning leverages likelihood-based methods which require precise knowledge of the cascade model \cite{NS12_cascades, ACFKP13_trace_complexity}.
\end{remark}

\begin{remark}
Though our main result concerns the number of cascades $K$ needed to learn high degree vertices, our techniques show that the infection times of high-degree vertices can be estimated from a single cascade; see Lemma \ref{lemma:single_cascade} in Section \ref{sec:alg_proof}.
\end{remark}

We note that Theorem \ref{thm:3/4} follows immediately from Theorem \ref{thm:alg_3/4} as a corollary. We provide a sketch of the proof of Theorem \ref{thm:alg_3/4} below, and defer the full details to Sections \ref{sec:alg_proof} and \ref{sec:second_derivative_proof}.

\begin{proof}[Proof sketch of Theorem \ref{thm:alg_3/4}]
Let $\gamma < \alpha$ be sufficiently close to $\alpha$. The strategy of our proof is to show, for all $v \in V$, that
\begin{equation}
\label{eq:degree_approximation}
\left| \widehat{\deg}_{\delta, i}(v) - \lambda \deg(v) \right| \le n^{\gamma }, 
\end{equation}
in all cascades. If \eqref{eq:degree_approximation} is true for a particular $i \in [K]$ and $v$ is a low-degree vertex, then $\widehat{\deg}_{\delta, i} (v) \le \lambda \deg(v) + n^{\gamma} \le 2n^\gamma$. Else if $v$ is high-degree, then $\widehat{\deg}_{\delta, i}(v) \ge \lambda \deg(v) - n^\gamma \ge \frac{\lambda}{2} n^\alpha$. If $\lambda$ is not known a priori, any choice of threshold satisfying $\tau = n^{\alpha - o(1)}$ works in distinguishing the two cases. In particular, $\tau = n^\alpha / \log(n)$ suffices.

Due to the way $\widehat{\deg}_{\delta, i}(v)$ is computed, establishing \eqref{eq:degree_approximation} reduces to showing that $I[t, t + \delta] / \delta$ concentrates around $J(t)$ and $I[t, t - \delta] / \delta$ concentrates around $J(t^-)$ for the values of $t$ corresponding to vertices' infection times.
For this task, choosing the right value of $\delta$ is crucial. If $\delta$ is too small, then $I [ t, t + \delta]$ may also be quite small, making meaningful concentration around its mean impossible. On the other hand, if $\delta$ is too large, the mean of $I[t, t + \delta ] / \delta$ may deviate significantly from $J(t)$ due to the potentially large changes to the set of infections in the interval $(t, t + \delta]$.
An additional complication arising if $\delta$ is too large is that
there may be a large number of low-degree vertices consistently infected around the same time as a high-degree vertex. If $u$ is one such low-degree vertex, then $\widehat{\deg}_\delta (u)$ can be quite large across multiple traces, leading to false positives. 

Examining these conditions in more detail exposes why we must assume $\alpha > 3/4$ for Algorithm \ref{alg:second_derivative} to succeed. Let us focus on the behavior of $I[t, t + \delta ] / \delta$ (the analysis of $I[t- \delta, t] / \delta$ is similar). We show that
\begin{equation}
\label{eq:first_derivative_concentration}
\frac{I[t, t + \delta]}{\delta} =  J(t) \pm O \left( n \delta  \right) \pm O \left( \sqrt{\frac{n}{\delta}} \right).
\end{equation}
Above, the $O(n \delta)$ fluctuation represents the deviation between $J(t)$ and the (conditional) mean of $I[t, t + \delta]/\delta$. This is essentially caused by the additional $O(n \delta)$ additional infections which occur in the interval $[t, t + \delta]$.
The $O( \sqrt{ n/\delta} )$ fluctuation comes from the deviation of $I[t, t + \delta]$ from its mean, which can be derived through subgaussian-type concentration inequalities.
To correctly estimate high-degree vertices through \eqref{eq:J}, both fluctuations need to be smaller than $D = n^\alpha$. We therefore require that $\max \{ n \delta, \sqrt{n / \delta} \} = o(n^\alpha)$, which is equivalent to
\begin{equation}
\label{eq:delta_lower_bound}
\delta = o \left( n^{- (1 - \alpha)} \right) \text{ and } \delta = \omega \left( n^{ - (2 \alpha - 1)} \right).
\end{equation}

\noindent We remark that \eqref{eq:delta_lower_bound} already requires $\alpha > 2/3$. Next, we need to ensure that there is no low-degree vertex $u$ that is infected around the same time as a high-degree vertex in all cascade traces. Otherwise, as we mentioned earlier, $\widehat{\deg}_{\delta,i}(u)$ could be quite large, which would imply that the output of Algorithm \ref{alg:second_derivative} contains false positives. With this scenario in mind, suppose that a low-degree $u$ becomes infected first at time $t$, and at that point in time, a high-degree $v$ has $M$ infected neighbors. By \eqref{eq:SI}, it holds for small $\delta$ that
\[
\p ( |T(u) - T(v) | \le \delta \vert \cF_t ) = \p ( v \in \cascade(t + \delta) \vert \cF_t ) = \delta M + o ( \delta).
\]
The probability that $u$ and $v$ are close in all cascade traces is therefore at most $O ( (\delta M)^{K})$. A union bound over pairs of low-degree $u$ and high-degree $v$ shows that
\begin{equation}
\label{eq:Tu_Tv_union_bound}
\p ( \exists u \in V, v \in \highdeg(G) : |T_i(u) - T_i(v) | \le \delta \text{ in all cascades} ) \le O \left( nm (\delta M)^{K} \right).
\end{equation}
To show that the right hand side of \eqref{eq:Tu_Tv_union_bound} is $o(1)$, we first need to choose $\delta$ such that $\delta M = o(1)$. 
Although $M$ can in principle be as large as $\deg(v)$ (i.e., if all of $v$'s neighbors are infected before $v$), we show that with high probability $M \le \sqrt{n}$ (see Lemma \ref{lemma:A} in Section \ref{sec:preliminaries}). Hence we may choose $\delta$ satisfying
\begin{equation}
\label{eq:delta_upper_bound}
\delta = o \left( n^{-1/2} \right).
\end{equation}
Writing $\delta = n^{- \beta}$ for $\beta > 0$, \eqref{eq:delta_lower_bound} and \eqref{eq:delta_upper_bound} imply that $\beta \in (1/2, 2 \alpha - 1)$. This interval is nonempty when $\alpha > 3/4$, and we may choose $\beta : = \alpha - 1/4$ to satisfy \eqref{eq:delta_lower_bound} and \eqref{eq:delta_upper_bound}. Finally, since $\delta = n^{- (\alpha - 1/4)}$ and $M \le \sqrt{n}$, $K > (\alpha - 3/4)^{-1}$ is needed for the right hand side of \eqref{eq:Tu_Tv_union_bound} to be $o(1)$.
\end{proof}

\subsection{Impossibility of estimation when $\alpha \in (0,1/2)$}
\label{sec:impossibility}

\paragraph{Connections to sparse mixture detection.} 
The cornerstone of our proof is a correspondence between high-degree estimation and the problem of sparse mixture detection. In a bit more detail, let $G'$ be a graph on $n / 2$ vertices of maximum degree $n^{o(1)}$ (i.e., all its vertices are low-degree). The graph $G$ will be formed by adding an additional $n/2$ vertices labelled $1, \ldots, n/2$ to $G'$, along with a single edge per added vertex which connects to a vertex in $G'$. A natural way of generating a graph $G$ \emph{without} any high-degree vertices is to let each of the new $n/2$ vertices connect to a uniform random parent in $G'$. Through standard probabilistic arguments, it can be seen that no more than $O ( \log n)$ neighbors are added to any particular vertex in $G'$, so $G$ contains no high-degree vertices. 

A simple modification of this procedure can also generate graphs with a high-degree vertex. Suppose we fix a vertex $v$ in $G'$. For each of the additional $n / 2$ vertices, with probability $2D / n$, it connects to $v$. Otherwise, with probability $1 - 2D / n$, it connects to a uniform random vertex in $G'$.
In expectation, $v$ receives $D$ additional connections (so, in particular, $\mathrm{deg}(v) \ge D$ in expectation). As before, all of the other vertices receive at most $O ( \log n)$ additional neighbors, with high probability. Hence we obtain a graph $G$ in which $v$ is the only high-degree vertex. 

Observe that in the second procedure, the distribution of the infection times of each of the added $n/2$ vertices can be viewed as a mixture of two distributions. To be precise, let $P$ be the distribution of a vertex's infection times in the observed cascade traces if they connect to a uniform random neighbor in $G'$. Similarly, let $Q$ be the distribution of a vertex's infection times if they connect to $v$. Then the task of differentiating between the two generative models for $G$ -- or equivalently, of testing whether $v$ is high-degree -- can be phrased as the following hypothesis testing problem:
\begin{align}
\label{eq:sparse_mixture_detection}
H_0 & : \mathbf{T}(1), \ldots \mathbf{T}(n/2) \sim P 
\hspace{1cm} H_1  : \mathbf{T}(1), \ldots, \mathbf{T}(n/2) \sim \left( \left( 1 - \epsilon_n \right) P + \epsilon_n Q  \right),
\end{align}
where we set $\epsilon_n : = 2D/n$ above, and we denote $\mathbf{T}(j) := (T_1(j), \ldots, T_K(j))$ where $T_i(j)$ is the infection time of vertex $j$ in the $i$th cascade.

It turns out that \eqref{eq:sparse_mixture_detection} is a special case of a well-studied problem known as \emph{sparse mixture detection}, which dates back to the work of Dobrusin \cite{dobrushin2958statistical}.
Burnashev \cite{burnashev1991problem} studied a variant related to Gaussian processes, and Ingster \cite{ingster1996some} identified a subtle phase transition between the possibility and impossibility of detection for mixtures of Gaussian distributions. The problem was popularized by the work of Donoho and Jin \cite{donoho2004higher}, who formulated an adaptive method in place of the likelihood ratio called \emph{higher criticism}, again for mixtures of Gaussians. Later, Cai and Wu \cite{cai2014optimal} developed an analysis of the sparse mixture problem for {generic} distributions. 
We remark that Mossel and Roch \cite{mossel2017distance} used sparse mixture detection in the inference of combinatorial structure, though in a very different context.

\paragraph{Impossibility of detecting high-degree vertices.} We solve our instance of the sparse mixture detection problem \eqref{eq:sparse_mixture_detection} using the techniques developed by Cai and Wu \cite{cai2014optimal}. That is, we bound the total variation distance between the two distributions in \eqref{eq:sparse_mixture_detection} by the $\chi^2$ divergence since it tensorizes and because it is easier to study for problems involving mixtures. Our analysis shows that if $K$ satisfies \eqref{eq:K_impossible}, then the $\chi^2$ divergence tends to zero. As a result, it is impossible to distinguish between scenarios where there are no high-degree vertices ($H_0$) and where $v$ is a high-degree vertex ($H_1$).

\paragraph{Putting everything together.} Finally, we show that the impossibility of detection implies that no estimator can successfully output the correct high-degree vertex with probability greater than $o(1)$. We remark that the ensemble $\cH$ and distribution $\mu$ stated in the theorem essentially corresponds to the random constructions for $G$ described in this subsection. For more details, see Section \ref{sec:impossibility_proof}.

\clearpage

\section{Preliminaries: Cascade dynamics and infection times}
\label{sec:preliminaries}

In this section we derive several useful results on the behavior of the SI process \eqref{eq:SI} and the distribution of infection times. Throughout, we will assume that $\cascade(t)$ is a realization of $\SI(G, v_0)$ for a particular graph $G \in \cG(n,m,d,D)$ and vertex $v_0 \in V$.

We begin by discussing an important equivalence between \eqref{eq:SI} and First Passage Percolation on $G$ with $\mathrm{Exp}(\lambda)$ edge weights. 
In a bit more detail, let $\{ \cF_t \}_{t \ge 0}$ denote the natural filtration corresponding to the SI process. Fix $t \ge 0$, and let $\{ F(e) \}_{e \in E}$ be a collection of i.i.d. $\mathrm{Exp}(\lambda)$ random variables that are independent of $\cF_t$. For a path\footnote{Recall that a path is a finite sequence of edges $e_1, \ldots, e_m$ such that for each $i \ge 1$, $e_i$ and $e_{i + 1}$ share exactly one endpoint.} $P$ in the graph, we define the \emph{weight} of the path to be $\weight(P) : = \sum_{e \in P} F(e)$. Then we have the following result.

\begin{proposition}
\label{prop:fpp}
Let $t \ge 0$, and condition on $\cF_t$. Suppose that $v \in V \setminus \cascade(t)$. Let $\cP(v,t)$ be the collection of paths from $\cascade(t)$ to $v$ where only the first edge in the path has an endpoint in $\cascade (t)$. Then the conditional distribution of $T(v)$ with respect to $\cF_t$ is given by 
\[
T(v) = t + \min_{P \in \cP(v,t)} \weight(P).
\]
\end{proposition}

This result is standard, and essentially follows from the memoryless property of exponential distributions. We defer the reader to \cite[Chapter 6]{fpp}, or to Section \ref{sec:fpp} for full details of the proof. 

Using the representation in Proposition \ref{prop:fpp}, we can derive a number of useful results on the infection times of vertices. We start by establishing probabilistic upper and lower bounds on conditional infection times of vertices.

\begin{lemma}
\label{lemma:infection_time_upper_bound}
Let $t \ge 0$, and condition on $\cF_t$. Suppose that $v \in V \setminus \cascade(t)$. Then $T(v) - t$ is stochastically dominated by a $\mathrm{Exp} ( \lambda | \cN(v) \cap \cascade(t) | )$ random variable.
\end{lemma}

\begin{proof}
Conditioned on $\cF_t$, Proposition \ref{prop:fpp} implies that
\begin{align*}
T(v) - t & = \min_{P  \in \cP(v,t) } \weight(P)  \le \min_{P \in \cP(v,t) : |P| = 1} \weight(P)  = \min_{u \in \cN(v) \cap \cascade(t) } F(u,v).
\end{align*}
Due to properties of exponential random variables, we have that 
\[
\min_{u \in \cN(v) \cap \cascade(t)} F(u,v) \sim \mathrm{Exp} \left( \lambda | \cN(v) \cap \cascade(t) | \right),
\]
which proves the claimed result.
\end{proof}

\begin{lemma}
\label{lemma:infection_time_lower_bound}
Let $t \ge 0$, condition on $\cF_t$, and suppose that $v \in V \setminus \cascade (t)$. Then $T(v) - t$ stochastically dominates a $\mathrm{Exp} ( \lambda | \cN(v)| )$ random variable.
\end{lemma}

\begin{proof}
Conditioned on $\cF_t$, we have that
\[
T(v) - t = \min_{P \in \cP(v,t)} \weight(P) \ge \min_{u \in \cN(v)} F(u,v),
\]
where the inequality follows since any path ending in $v$ must include one of the edges incident to $v$. The claim follows since $\min_{u \in \cN(v) } F(u,v) \sim \mathrm{Exp} ( \lambda | \cN(v) | )$.
\end{proof}

We next define a useful event concerning the number of infected neighbors of high-degree vertices. We remark that the control of this quantity is especially important for us, as it is equal to the bias of our estimate of vertex degrees (see \eqref{eq:J}).

\begin{definition}
\label{def:A}
The event $\cA$ holds if and only if 
\[
| \cN(v) \cap \cascade (T(v)) | \le 2d \sqrt{n} \log^2 n,
\]
for all $v \in V$.
\end{definition}

\begin{remark}
As we assume throughout this paper that $d = n^{o(1)}$, we will often use that $| \cN(v) \cap \cascade (T(v)) | \le n^{1/2 + o(1)}$ when using the event $\cA$.
\end{remark}

\begin{remark}
Assuming that $G \in \cG(n,m,d,D)$, the event $\cA$ is only relevant for high-degree vertices, since 
\[
| \cN(v) \cap \cascade ( T(v)) | \le \deg(v) = n^{o(1)} \le 2d \sqrt{n} \log^2 n,
\]
otherwise.
\end{remark}

\begin{lemma}
\label{lemma:A}
Suppose that $m = o (\log^2 n)$. Then $\p ( \cA) = 1 - o(1)$. 
\end{lemma}

\begin{proof}
If $\deg(v) \le 2d \sqrt{n} \log^2 n$, then 
\[
| \cN(v) \cap \cascade (T(v)) | \le | \cN(v) | \le 2d \sqrt{n} \log^2 n,
\]
which satisfies the condition in $\cA$ corresponding to $v$. It therefore suffices to consider the case $\deg(v) \ge 2d \sqrt{n} \log^2 n$. Define the event 
\[
\cE_v : = \left \{ | \cN(v) \cap \cascade (T(v)) | \ge 2d \sqrt{n} \log^2 n \right \},
\]
and let $T$ be the stopping time at which $\sqrt{n} \log n$ neighbors of $v$ are first infected. 

Let $t \ge 0$, and condition on $\cF_t$. Let us also assume that $T = t$ (an event that is $\cF_t$-measurable), so that $| \cN(v) \cap \cascade (t) | = \sqrt{n} \log n$. If $T(v) \le t$, then we have that $\p ( \cE_v \vert \cF_t ) = 0$. We therefore proceed by considering the case where $T(v) > t$. Notice that if $T = t$,
\[
| \cN(v) \cap \cascade(T(v)) | = | \cN(v) \cap \cascade(t) | + | \cN(v) \cap \cascade[t, T(v)] | = \sqrt{n} \log n + | \cN(v) \cap \cascade[t, T(v)] |.
\]
As a result, we have that
\begin{equation}
\label{eq:Ev_bound_1}
\p ( \cE_v \vert \cF_t )\mathbf{1}(T = t) \le \p ( | \cN(v) \cap \cascade[t, T(v)] | \ge d \sqrt{n} \log^2 n \vert \cF_t ) \mathbf{1}(T = t).
\end{equation}
To bound the probability on the right hand side, we start by bounding $T(v)$.
Lemma \ref{lemma:infection_time_upper_bound} implies that $T(v) - t$ is stochastically dominated by a $\mathrm{Exp}( \lambda \sqrt{n} \log n)$ random variable. Hence
\begin{equation}
\label{eq:T_v_bound}
\p \left( \left. T(v) \ge t + \frac{1}{\lambda \sqrt{n}} \right \vert \cF_t \right) \le \frac{1}{n}.
\end{equation}
We can now bound the probability of interest as follows.
\begin{align}
\p ( | \cN(v) \cap \cascade[t, T(v)] | \ge d \sqrt{n } \log^2 n | \cF_t ) & \stackrel{(a)}{\le} \p (T(v) \ge t' \vert \cF_t ) + \p ( | \cN(v) \cap \cascade[t, t'] | \ge d \sqrt{n} \log^2 n \vert \cF_t ) \nonumber \\
\label{eq:Ev_bound_2}
& \stackrel{(b)}{\le} \frac{1}{n} + \frac{ \E [ | \cN(v) \cap \cascade[t, t'] | \vert \cF_t]}{d \sqrt{n} \log^2 n}.
\end{align}
Above, $(a)$ is due to a union bound, and since $| \cN(v) \cap \cascade[t, T(v)] | \le | \cN(v) \cap \cascade[t, t']|$ whenever $T(v) < t'$. The inequality $(b)$ is due to \eqref{eq:T_v_bound} and Markov's inequality.

We proceed by bounding $\E [ | \cN(v) \cap \cascade[t, t'] | \vert \cF_t ]$; to do so, we study the probability that a given $u \in \cN(v) \setminus \cascade(t)$ is infected in the interval $[t, t']$. If $u \in \highdeg(G)$, we trivially bound the probability by 1. If $u \notin \highdeg(G)$, $\deg(u) \le d$ so by Lemma \ref{lemma:infection_time_lower_bound}, $T(u) - t$ stochastically dominates a $\mathrm{Exp} ( \lambda d)$ random variable. As a result, for such $u$, we have that
\[
\p ( T(u) \in [t, t'] \vert \cF_t ) \le 1 - e^{- \lambda (t' - t)} \le \frac{d}{\sqrt{n}}.
\]
In the final inequality above, we have used $1 - e^{-x} \le x$ as well as the definition of $t'$. It follows that
\begin{equation}
\label{eq:Ev_bound_3}
\E [ | \cN(v) \cap \cascade[t, t'] | \vert \cF_t] = \sum_{u \in \cN(v) \setminus \cascade(t)} \p ( T(u) \in [t, t'] \vert \cF_t ) \le m + d \sqrt{n} \le 2 d \sqrt{n}.
\end{equation}
In the first inequality above, we have used that there are at most $m$ high-degree vertices in $\cN(v)$ and at most $n$ low-degree vertices.
Together, \eqref{eq:Ev_bound_1}, \eqref{eq:Ev_bound_2} and \eqref{eq:Ev_bound_3} show that 
\[
\p ( \cE_v \vert \cF_t ) \mathbf{1}( T = t) \le \frac{3}{\log^2 n} \mathbf{1}(T = t).
\]
Since $T < \infty$ almost surely (since all vertices become infected in finite time), it holds that
\[
\p ( \cE_v \vert \cF_T ) = \int_0^\infty \p ( \cE_v \vert \cF_t ) \mathbf{1}(T = t) dt \le \frac{3}{\log^2 n}.
\]
Taking an expectation over $\cF_T$ shows that $\p ( \cE_v ) \le 3/ \log^2 (n)$. Finally, we note that 
\[
\cA^c = \bigcup_{v \in V: \deg(v) \ge 2d \sqrt{n} \log^2 n} \cE_v,
\]
hence a union bound over these at most $m$ events shows that
\[
\p( \cA^c) \le \frac{3m}{\log n} = o(1).
\]
\end{proof}

As the following result shows, the event $\cA$ is crucially used to obtain sharper characterizations of infection times than what Lemma \ref{lemma:infection_time_lower_bound} can provide, especially for high-degree vertices.

\begin{lemma}
\label{lemma:T_v_small_interval}
Let $v \in V$ and let $\delta > 0$ satisfy $\lambda^2 \delta^2 m n \le 1$. Fix $t \ge 0$, and let $v \in V \setminus \cascade(t)$. Then 
\[
\p \left(\left.  \{ T(v) \in [t, t + \delta ] \} \cap \cA \right \vert \cF_t \right) \le 3d \lambda \delta \sqrt{n} \log^2 (n).
\]
\end{lemma}

\begin{proof}
Define the event
\[
\cA_t : = \left \{ \text{For all $v \in V$ such that $T(v) > t$, it holds that $| \cN(v) \cap \cascade (t) | \le 2d \sqrt{ n } \log^2 n$} \right \}.
\]
Notice that $\cA_t$ is $\cF_t$-measurable, and $\cA \subset \cA_t$. Moreover, for $v$ such that $T(v) > t$,
\begin{equation}
\label{eq:T_v_E_bound}
\p \left( \{ T(v) \le t + \delta \} \cap \cA \vert \cF_t \right) \le \p \left( \{ T(v) \le t + \delta \} \cap \cA_t \vert \cF_t \right) = \p ( T(v) \le t + \delta \vert \cF_t ) \mathbf{1}(\cA_t).
\end{equation}
We will therefore proceed by bounding $\p ( T(v) \le t + \delta \vert \cF_t)$ under the assumption that $\cA_t$ holds. 

Suppose that $\deg(v) \le d$. Then by Lemma \ref{lemma:infection_time_lower_bound}, $T(v) - t$ stochastically dominates a $\mathrm{Exp}( \lambda d)$ random variable, so that
\begin{equation}
\label{eq:T_v_bound_low_degree}
\p ( T(v) \le t + \delta \vert \cF_t ) \le 1 - e^{ - \delta \lambda d} \le \delta \lambda d.
\end{equation}
Next, suppose that $\deg(v) \ge D$, and let $\cP_2(v)$ be the set of paths of length 2 starting from $v$. Then by Proposition \ref{prop:fpp},
\begin{align}
\label{eq:T_v_lower_bound_1}
T(v) - t & = \min_{P \in \cP(v,t) } \weight(P) \\
\label{eq:T_v_lower_bound_2}
& \ge \min \left \{ \min_{u \in \cN(v) \cap \cascade (t) } F(u,v) , \min_{P \in \cP_2(v) } \weight(P) \right \},
\end{align}
which can be explained as follows. If the minimizing path $P \in \cP(v,t)$ in \eqref{eq:T_v_lower_bound_1} has length 1, then the minimum is achieved by $F(u,v)$ for some $u \in \cN(v) \cap \cascade (t)$. Otherwise, the minimizing path in \eqref{eq:T_v_lower_bound_1} has length at least 2, which must include a path of length 2 starting at $v$. In light of \eqref{eq:T_v_lower_bound_2},
\begin{align}
\p ( T(v) \le t + \delta \vert \cF_t ) & \le \p \left( \left. \min \left \{ \min_{u \in \cN(v) \cap \cascade (t) } F(u,v) , \min_{P \in \cP_2(v) } \weight(P) \right \} \le \delta \right \vert \cF_t \right) \nonumber \\
& \le \sum_{u \in \cN(v) \cap \cascade (t) } \p ( F(u,v) \le \delta  \vert \cF_t ) + \sum_{P \in \cP_2(v) } \p ( \weight(P) \le \delta \vert \cF_t ) \nonumber \\
& \stackrel{(a)}{\le} | \cN(v) \cap \cascade (t) | \lambda \delta + | \cP_2(v) | \lambda^2 \delta^2 \nonumber \\
& \stackrel{(b)}{\le} 2d \lambda \delta \sqrt{n} \log^2 (n) + (dn + mn) \lambda^2 \delta^2  \nonumber \\
\label{eq:T_v_bound_high_degree}
& \stackrel{(c)}{\le} 3d \lambda \delta \sqrt{n} \log^2 (n),
\end{align}
where the inequality $(a)$ follows since $F(u,v) \sim \mathrm{Exp}(\lambda)$ and $\weight(P) \sim \mathrm{Gamma}(2, \lambda)$ for $P \in \cP_2(v)$. The inequality $(b)$ uses the bound $| \cN(v) \cap \cascade (t) | \le 2d \sqrt{n} \log^2 n$ which holds on the event $\cA_t$, and also uses that the number of paths of length 2 starting from $v$ is at most $dn + mn$, since there are at most $n$ neighbors of degree at most $d$ and at most $m$ neighbors of degree at most $n$.
Finally, $(c)$ holds by our assumption that $\lambda^2 \delta^2 m n \le 1$.
\end{proof}

A useful consequence of Lemma \ref{lemma:T_v_small_interval} is a bound on the probability that $T(u)$ and $T(v)$ are close, for distinct vertices $u,v$.

\begin{lemma}
\label{lemma:temporal_separation}
Let $u,v \in V$ be distinct vertices and let $\delta > 0$ satisfy $\lambda^2 \delta^2 n \le 1$. Then 
\[
\p ( \{ |T(u) - T(v) | \le \delta \} \cap \cA ) \le 6d \lambda \delta \sqrt{n} \log^2(n).
\]
\end{lemma}

\begin{proof}
Suppose without loss of generality that $T(u) < T(v)$, so that $|T(u) - T(v) | \le \delta$ is equivalent to $T(v) \in [T(u), T(u) + \delta]$. Then by Lemma \ref{lemma:T_v_small_interval}, 
\begin{align*}
\p ( \{ T(v) \in [T(u), T(u) + \delta] \} \cap \cA \vert \cF_t ) \mathbf{1}( T(u) = t) & = \p ( \{ T(v) \in [t, t + \delta] \} \cap \cA \vert \cF_t ) \mathbf{1}(T(u) = t) \\
& \le 3d \lambda \delta \sqrt{n} \log^2 (n).
\end{align*}
Since $T(u) < \infty$ almost surely, we have that 
\begin{align*}
\p ( \{ T(v) \in [T(u), T(u) + \delta] \} \cap \cA \vert \cF_{T(u)} ) & = \int_0^\infty \p ( \{ T(v) \in [T(u), T(u) + \delta] \} \cap \cA \vert \cF_t ) \mathbf{1}( T(u) = t) dt \\
& \le 3d \lambda \delta \sqrt{n} \log^2 (n).
\end{align*}
Finally, the claim follows from taking an expectation over $\cF_{T(u)}$.
\end{proof}

\begin{lemma}
\label{lemma:infections_poisson_bound}
For any $0 \le t_1 < t_2 <\infty$, $| \cascade(t_2) \setminus \cascade(t_1) |$ is stochastically bounded by a $\mathrm{Poi} ( \lambda mnd (t_2 - t_1))$ random variable.
\end{lemma}

\begin{proof}
Recall that $|E|$ is the number of edges in the graph, and let $A_1, A_2, \ldots$ be a sequence of i.i.d. $\mathrm{Exp}(\lambda |E|)$ random variables. As a shorthand, for a positive integer $k$, let $\tau_k$ denote the time at which the $k$th vertex is infected, after time $t_1$. We claim that it suffices to show that, conditioned on $\cF_{t_1}$, 
\begin{equation}
\label{eq:infection_time_stoch_lower_bound}
\tau_k - t_1 \succcurlyeq \sum_{i = 1}^k A_i.
\end{equation}
With \eqref{eq:infection_time_stoch_lower_bound} in hand, it is readily seen that the arrival times of vertices after time $t_1$ stochastically bounds the arrival times of a homogeneous Poisson point process with rate $\lambda |E|$. As a result, $| \cascade(t_2) \setminus \cascade(t_1)|$ is stochastically bounded by a $\mathrm{Poi}( \lambda |E| (t_2 - t_1))$ random variable. Finally, the claim follows since we can upper bound 
\[
|E| \le \sum_{v \in V} \deg(v) \le nd + mn \le mnd,
\]
where the second inequality above follows since there are at most $n$ vertices of degree at most $d$ and there are $m$ vertices of degree at most $n$.

It remains to prove \eqref{eq:infection_time_stoch_lower_bound}; we do so by induction. Let $E(t_1) \subset E$ be the set of edges where one endpoint is in $\cascade(t)$ and the other is not in $\cascade(t)$. Since the next infection must be a vertex which neighbors a currently infected vertex, we have by Proposition \ref{prop:fpp} that $\tau_1 - t_1 \ge \min_{e \in E(t_1)} F(e)$, where we recall that $\{F(e) \}_{e \in E}$ is a collection of i.i.d. $\mathrm{Exp}(\lambda)$ random variables. Notice that by properties of Exponential random variables, $\min_{e \in E(t_1)} F(e)$ is equal in distribution to a $\mathrm{Exp}(\lambda |E(t_1)|)$ random variable, which in turn stochastically bounds a $\mathrm{Exp}(\lambda |E| )$ random variable. This proves the base case. 

Now let $k \ge 1$, and assume that \eqref{eq:infection_time_stoch_lower_bound} holds for $k$. We may assume without loss of generality that $\tau_{k + 1} < \infty$ almost surely. Otherwise, since $G$ is connected, this must mean that there are no more infection events that can happen after the $k$th infection. In this case, $\tau_{k + 1} = \infty$, hence \eqref{eq:infection_time_stoch_lower_bound} trivially holds with $k$ replaced with $k + 1$. Let us now proceed with the assumption that $\tau_{k + 1} < \infty$ almost surely. Fix $s \ge t_1$, condition on $\cF_s$, and let us assume that $\tau_k = s$ (an event that is $\cF_s$-measurable). Repeating the same arguments above, we see that, conditioned on $\cF_s$ and assuming $\tau_k = s$, we have that $\tau_{k + 1} - s$ stochastically dominates a $\mathrm{Exp}(\lambda |E|)$ random variable. Hence, letting $x \ge 0$ and $A \sim \mathrm{Exp}( \lambda |E|)$, we have that
\[
\p ( A \ge x ) \mathbf{1}(\tau_k = s) \le \p ( \tau_{k + 1} - s \ge x \vert \cF_s ) \mathbf{1}(\tau_k = s) = \p ( \tau_{k + 1} - \tau_k \ge x \vert \cF_s ) \mathbf{1}(\tau_k = s).
\]
Taking an integral over $0 \le s < \infty$, we have that
\[
\p ( A \ge x) \le \int_0^\infty \p ( \tau_{k + 1} - \tau_k \ge x \vert \cF_s ) \mathbf{1} ( \tau_k = s) ds = \p ( \tau_{k + 1} - \tau_k \ge x \vert \cF_{\tau_k}).
\]
In particular, $\tau_{k + 1} - \tau_k \succcurlyeq A$ conditioned on $\tau_k$. As $A$ is independent of $\tau_k$ and \eqref{eq:infection_time_stoch_lower_bound} holds, it immediately follows that \eqref{eq:infection_time_stoch_lower_bound} also holds with $k$ replaced with $k + 1$. 
\end{proof}

\section{Analysis of Algorithm 1: Proof of Theorem \ref{thm:alg_3/4}}
\label{sec:alg_proof}

The following lemma, which characterizes when $\widehat{\deg}_\delta(v)$ is a good approximation for $\deg(v)$, is the main result of this section. 

\begin{lemma}
\label{lemma:degree_concentration}
Suppose that $\beta, \gamma \in (0,1)$ satisfy $\gamma > \max \{ 1 - \beta, (1 + \beta) / 2, \beta \}$, and let $\delta : = n^{-\beta}$. It holds with probability $1 - o(1)$ that for all $v \in V$, at least one of the following conditions is true:
\begin{enumerate}
    \item[(C1)] $\left| \widehat{\deg}_\delta(v) - \lambda \deg(v) \right| \le n^\gamma$.
    \item [(C2)] There exists $u \in \highdeg(G) \setminus \{v \}$ such that $|T(v) - T(u) | < n^{-\beta}$.
\end{enumerate}
\end{lemma}

Before proving Lemma \ref{lemma:degree_concentration}, we show how it can be used to prove Theorem \ref{thm:alg_3/4}. We first analyze what can be learned from a single cascade.

\begin{lemma}
\label{lemma:single_cascade}
Suppose that $\alpha, \beta \in (0,1)$ satisfy $\alpha > \max \{ 1 - \beta, (1 + \beta ) / 2 , \beta \}$ and $\beta > 1/2$.
Define 
\begin{equation}
\label{eq:HD_estimate}
S : = \left \{ v : \widehat{\deg}(v) \ge \frac{n^\alpha}{\log n} \right \}.
\end{equation}
Then with probability $1 - o(1)$,
\begin{equation}
\label{eq:high_degree_inclusions}
\highdeg(G) \subseteq S \subseteq \left \{ u \in V : \exists v \in \highdeg(G) \text{ with } |T(u) - T(v) | \le \frac{1}{n^\beta} \right \}.
\end{equation}
\end{lemma}

\begin{proof}
Throughout this proof, we will assume that the event described in Lemma \ref{lemma:degree_concentration} holds. We may do so since this event holds with probability $1 - o(1)$.

We start by proving that the first inclusion in \eqref{eq:high_degree_inclusions} holds with probability $1 - o(1)$. To this end, define the event 
\[
\cE : = \left \{ \text{For all distinct $u,v \in \highdeg(G)$, it holds that $| T(u) - T(v) | > \frac{1}{n^\beta}$} \right \}.
\]
Suppose that $\cE$ holds, and choose $\gamma$ such that $\alpha > \gamma > \max \{ 1 - \beta, (1 + \beta)/2, \beta \}$.
Then for all $v \in \highdeg(G)$, we must have that $| \widehat{\deg}(v) - \lambda \deg(v) | \le n^\gamma$, since the condition (C2) in Lemma \ref{lemma:degree_concentration} cannot hold on $\cE$ as $v \in \highdeg(G)$. Hence
\begin{equation}
\label{eq:estimated_degree_lower_bound}
\widehat{\deg}(v) \ge \lambda \deg(v) - n^\gamma \ge \lambda n^\alpha - n^\gamma \ge \frac{n^\alpha}{\log n},
\end{equation}
where we have used the fact that $\lambda$ is constant with respect to $n$, and that $\alpha > \gamma$. Moreover, an important consequence of \eqref{eq:estimated_degree_lower_bound} is that $\highdeg(G) \subseteq S$. To complete the proof of the first inclusion in \eqref{eq:high_degree_inclusions}, we need to show that $\cE$ holds with probability $1 - o(1)$. To this end, let us recall the event $\cA$ (see Definition \ref{def:A}). We can bound
\begin{equation}
\label{eq:E_cap_A}
\p ( \cE^c \cap \cA) \le \sum_{u,v \in \highdeg(G) : u \neq v} \p \left( \left \{  | T(u) - T(v) | \le \frac{1}{n^\beta} \right \} \cap \cA \right) \stackrel{(a)}{\le} n^{- (\beta - 1/2) + o(1)} \stackrel{(b)}{=} o(1).
\end{equation}
Above, $(a)$ follows since there are at most $m^2 = n^{o(1)}$ terms in the summation, and each term in the summation can be bounded via Lemma \ref{lemma:temporal_separation}. The equality $(b)$ holds under our assumption that $\beta > 1/2$. By a union bound, we have that
\[
\p ( \cE^c) \le \p ( \cE^c \cap \cA ) + \p ( \cA^c) = o(1),
\]
where the bound on the right hand side is due to \eqref{eq:E_cap_A} and Lemma \ref{lemma:A}.

We now turn to the proof of the second inclusion in \eqref{eq:high_degree_inclusions}. As a shorthand, let us denote 
\[
S' : = \left \{ u \in V : \exists v \in \highdeg(G) \text{ with } | T(u) - T(v) | \le \frac{1}{n^\beta} \right \}.
\]
Let $u \in S$. If $u \in \highdeg(G)$, then we trivially have that $u \in S'$. On the other hand, if $u \notin \highdeg(G)$, then $\deg(u) \le n^{o(1)}$ so, under condition (C1) of Lemma \ref{lemma:degree_concentration}, it cannot be the case that $u \in S$. Thus, (C2) must hold, which implies that $u \in S'$ in this case as well. Both cases prove that $S \subseteq S'$.
\end{proof}

We are now ready to prove the main result of this section.

\begin{proof}[Proof of Theorem \ref{thm:alg_3/4}]
Let $S_i$ denote the version of the set in \eqref{eq:HD_estimate} corresponding to the $i$th cascade, for $i \in [K]$. In particular, we have that $\HDestimate = \bigcap_{i= 1}^K S_i$, where $\HDestimate$ is the output of Algorithm \ref{alg:second_derivative}.
With this notation in hand, it suffices to show that $\bigcap_{i = 1}^K S_i = \highdeg(G)$ with probability $1 - o(1)$.
By Lemma \ref{lemma:single_cascade}, we have with probability $1 - o(1)$ that $\bigcap_{i = 1}^K S_i \supseteq \highdeg(G)$, provided $K$ is constant with respect to $n$. In the remainder of the proof, we therefore focus on proving that $\bigcap_{i = 1}^K S_i \subseteq \highdeg(G)$ with high probability. 
It will be useful to define $\cA_i$ to be the event described in Definition \ref{def:A} which concerns the $i$th cascade, for each $i \in [K]$.

We proceed by bounding the probability that there exists a vertex $u$ satisfying $u \notin \highdeg(G)$ and $u \in S_i$, for a fixed $i \in [K]$. If there exists such a vertex, then by the characterization in Lemma \ref{lemma:single_cascade}, there must exist $v \in \highdeg(G)$ such that $|T_i(u) - T_i(v) | \le n^{-\beta}$. We have that
\begin{align*}
\p \left ( \left \{ u \in S_i \right \} \cap \cA_i \right ) & \le \p \left( \left \{ \exists v \in \highdeg(G) : |T_i(u) - T_i(v) | \le \frac{1}{n^\beta} \right \} \cap \cA_i \right) \\
& \le \sum_{v \in \highdeg(G)} \p \left( \left \{ |T_i(u) - T_i(v) | \le \frac{1}{n^\beta} \right \} \cap \cA_i \right) \\
& \le n^{- ( \beta - 1/2) + o(1)}.
\end{align*}
Above, the final inequality follows from Lemma \ref{lemma:temporal_separation}, and since $|\highdeg(G) | \le m = n^{o(1)}$. Next, by the independence of cascades, we have that for $u \notin \highdeg(G)$,
\[
\p \left( \left \{ u \in \bigcap_{i = 1}^K S_i \right \} \cap \bigcap_{i = 1}^K \cA_i \right) = \prod_{i = 1}^K \p \left( \left \{ u \in S_i \right \} \cap \cA_i \right) = n^{- K ( \beta - 1/2) + o(1)}.
\]
Putting everything together, we have that
\begin{align*}
\p \left( \left \{ \left| \left( \bigcap_{i = 1}^K S_i \right) \setminus \highdeg(G) \right| \ge 1 \right \} \cap \bigcap_{i = 1}^K \cA_i \right) & \le \E \left [ \left| \left( \bigcap_{i = 1}^K S_i \right) \setminus \highdeg(G) \right| \mathbf{1} \left( \bigcap_{i = 1}^K \cA_i \right) \right ] \\
& = \sum_{u \notin \highdeg(G)} \p \left( \left \{ u \in \bigcap_{i = 1}^K S_i \right \} \cap \bigcap_{i = 1}^K \cA_i \right) \\
& = n^{1 - K ( \beta - 1/2) + o(1)} \\
& = o(1),
\end{align*}
where $(a)$ is due to Markov's inequality, and $(b)$ holds whenever $K > 1 / (\beta - 1/2)$. Finally, the theorem follows from setting $\beta = \alpha - 1/4$, and by noting that $\p ( \bigcap_{i = 1}^K \cA_i ) = 1 - o(1)$ in light of Lemma \ref{lemma:A}.
\end{proof}

In the rest of this section, we will focus on proving Lemma \ref{lemma:degree_concentration}. For our analysis, it will be convenient to define the quantities
\begin{align}
\label{eq:empirical_derivative}
\hatderivative(t) &: = n^\beta  \left( I \left[ t, t + \frac{1}{n^\beta} \right] - I \left[t - \frac{1}{n^\beta}, t \right] \right) \\
\label{eq:expected_derivative}
\derivative(t) & : = \lambda \left( \cut ( \cascade(t)) - \cut ( \cascade(t^-)) \right).
\end{align}
In words, $\hatderivative(t)$ and $\derivative(t)$ are the discretized and expected second derivatives of the infection curve, respectively. Recall that we can relate these to the quantities of interest via $\widehat{\deg}_\delta(v) = \hatderivative(T(v))$ for $\delta = n^{-\beta}$ and $\derivative(T(v)) = \lambda ( \deg(v)- 2 | \cN(v) \cap \cascade(T(v)) | )$.

Our first intermediate result towards proving Lemma \ref{lemma:degree_concentration} controls the difference between $\hatderivative(t)$ and $\derivative(t)$ for a \emph{fixed} $t \ge 0$.

\begin{lemma}
\label{lemma:second_derivative_approximation}
Let $\beta, \gamma \in (0,1)$ satisfy $\gamma > \max \{ 1 - \beta, (1 + \beta) / 2 \}$. Let us also define the event
\begin{equation}
\label{eq:Bt}
\cB_t : = \left \{ \forall v \in \highdeg(G), T(v) \notin \left(t - \frac{1}{n^\beta}, t \right) \cup \left( t, t + \frac{1}{n^\beta} \right) \right \}.
\end{equation}
Then for any $t \ge 0$, it holds that
\[
\p \left( \left \{  \left | \hatderivative(t) -  \derivative(t) \right| \ge  n^\gamma \right \} \cap \cB_t \right) \le \exp \left( - n^{2 \gamma - (1 + \beta) + o(1)} \right).
\]
\end{lemma}

Since the proof of the lemma is somewhat involved, we defer the details to Section \ref{sec:second_derivative_proof}.

Our next goal is to show that $\hatderivative(t)$ and $\derivative(t)$ are close for \emph{all} $t \ge 0$ whenever $\cB_t$ holds, rather than for just a fixed $t \ge 0$. To prove this, we will employ a discretization argument. Define the finite set of time indices
\[
\mathbb{T}_n : = \left \{ \frac{k}{n^5} : k \in \mathbb{Z}_{\ge 0} \right \} \cap [0,n^2].
\]
The following result shows that the discretization $\mathbb{T}_n$ is rich enough to capture the continuous-time dynamics of the cascade.

\begin{lemma}
\label{lemma:discretization}
With probability $1 - o(1)$, for each $v \in V$ there exists a unique $t \in \mathbb{T}_n$ such that $t - 1/n^5 \le T(v) \le t$.
\end{lemma}

\begin{proof}
We first show that $\max_{v \in V} T(v) \le n^2$. From the representation in Proposition \ref{prop:fpp}, we can upper bound $\max_{v \in V} T(v)$ by $\sum_{e \in E} F(e)$, so we focus on characterizing the latter quantity. In particular, we have that
\[
\E \left[ \sum_{e \in E} F(e) \right] = \frac{|E|}{\lambda}.
\]
We can upper bound the expectation by noting that $|E| \le dn + mn = n^{1 + o(1)}$, which holds since there are at most $n$ vertices of degree  at most $d$ and at most $m$ vertices of degree at most $n$ for $G \in \cG(n,m,d,D)$. Markov's inequality now implies
\[
\p \left( \max_{v \in V} T(v) \ge n^2 \right) \le n^{-1 + o(1)}.
\]
Hence there must exist some interval of the form $[t, t + 1/n^5]$ which contains $T(v)$, for every $v \in V$.

We next show that in each interval of the form $[t, t + 1/n^5]$, there exists at most one infection event. Condition on $\cF_t$. If there are at least 2 infection events in $[t, t + 1/n^5]$, then there must exist distinct $e_1, e_2 \in E$ such that $F(e_1), F(e_2) \le n^{-5}$ (if not, at most one vertex is infected in the interval). The probability of this occurring is at most 
\[
{|E| \choose 2} \left( 1 - e^{ - \lambda / n^5} \right)^2 \le n^{2 + o(1)} \left(\frac{\lambda}{n^5} \right)^2 = n^{-8 + o(1)}.
\]
We conclude by taking a union bound over the $O(n^7)$ elements of $\mathbb{T}_n$.
\end{proof}

Together, Lemmas \ref{lemma:second_derivative_approximation} and \ref{lemma:discretization} prove Lemma \ref{lemma:degree_concentration}.

\begin{proof}[Proof of Lemma \ref{lemma:degree_concentration}]
We start by relating maximal deviations between $\widehat{\deg}_\delta(v)$ and $\lambda \deg(v)$ to maximal deviations between $\hatderivative(t)$ and $\derivative(t)$. 
Recall that $\hatderivative(T(v)) = \widehat{\deg}_\delta(v)$ for $\delta = n^{-\beta}$ and that $\derivative(T(v)) = \lambda ( \deg(v) - 2 | \cN(v) \cap \cascade(T(v)) |)$. By the triangle inequality, we have that
\begin{align*}
 \left| \widehat{\deg}_\delta(v) - \lambda \deg(v) \right|  & \le  \left| \hatderivative(T(v)) - \derivative(T(v)) \right| + 2 \lambda \left| \cN(v) \cap \cascade(T(v)) \right |.
\end{align*}
Recall the event $\cA$ from Definition \ref{def:A} as well as the event $\cB_t$ from \eqref{eq:Bt}. On $\cA$, it holds that
\begin{align}
\max_{v \in V} \left| \widehat{\deg}_\delta(v) - \lambda \deg(v) \right | \mathbf{1}( \cB_{T(v)} ) & \le \max_{v \in V} \left| \hatderivative(T(v)) - \derivative(T(v)) \right| \mathbf{1}( \cB_{T(v)} ) + 2 \lambda \max_{v \in V} \left | \cN(v) \cap \cascade(T(v)) \right| \nonumber \\
\label{eq:degree_derivative_deviation}
& \le \sup_{t \ge 0} \left| \hatderivative(t) - \derivative(t) \right| \mathbf{1}( \cB_t ) + 4 \lambda d \sqrt{n} \log^2 n.
\end{align}
In light of \eqref{eq:degree_derivative_deviation}, we have for any $\gamma$ satisfying the conditions of the lemma that
\begin{align}
\p \left( \max_{v \in V} \left| \widehat{\deg}_\delta(v) - \lambda \deg(v) \right| \mathbf{1}( \cB_{T(v)} ) \ge n^\gamma \right) & \le \p ( \cA^c ) + \p \left( \sup_{t \ge 0} \left| \hatderivative(t) - \derivative(t) \right| \mathbf{1}( \cB_t ) \ge n^\gamma - 4 \lambda d \sqrt{n} \log^2 n \right) \nonumber \\
\label{eq:degree_deviation_bound}
& \le o(1) + \p \left( \sup_{t \ge 0} \left| \hatderivative(t) - \derivative(t) \right| \mathbf{1}( \cB_t ) \ge n^{\gamma'} \right).
\end{align}
In the second line, the $o(1)$ bound is due to Lemma \ref{lemma:A}, and we have used the fact that $\gamma > 1/2$ to argue that $n^\gamma - 4\lambda d \sqrt{n} \log^2 n \ge n^{\gamma'}$ for some $\gamma'<\gamma$ satisfying the same conditions as $\gamma$ in the lemma statement. In the remainder of the proof, we show that the probability on the right hand side in \eqref{eq:degree_deviation_bound} is also $o(1)$. Once this is shown, the proof is concluded, since the event 
\[
\left \{ \max_{v \in V} \left| \widehat{\deg}_\delta(v) - \lambda \deg(v) \right| \mathbf{1}( \cB_{T(v)} ) \le n^\gamma \right \}
\]
is equivalent to the one described in the statement of Lemma \ref{lemma:degree_concentration}.

We therefore proceed by upper bounding the probability on the right hand side of \eqref{eq:degree_deviation_bound}.
Assume that the event described in the statement of Lemma \ref{lemma:discretization} holds. 
Fix $t \ge 0$, and let $t' \in \mathbb{T}_n$ satisfy $t' \le t < t' + n^{-5}$. We first claim that $\derivative(t) = \derivative(t')$ or $\derivative(t) = \derivative(t' + n^{-5})$. To see why, we consider three cases. In the first case, suppose that no vertex is infected in the interval $[t', t' + n^{-5})$. Then there is no change in $\cascade(t)$ across the entire interval, hence $\derivative(t) = \derivative(t') = \derivative(t' + n^{-5})$. In the second case, suppose that there is a single infection event in the interval, which occurs in the sub-interval $[t', t)$. Then there is no change in $\cascade(t)$ in the interval $[t, t' + n^{-5}]$, hence $\derivative(t) = \derivative(t' + n^{-5})$. Finally, in the third case, suppose again that there is a single infection event in the interval, which occurs in the sub-interval $[t, t' + n^{-5})$. By analogous arguments, it follows that $\derivative(t) = \derivative(t')$. There are no other cases since, on the event described in the statement of Lemma \ref{lemma:discretization}, at most one infection event can occur in $[t', t' + n^{-5})$.

Next, we study $\widehat{\derivative}(t)$. By the triangle inequality, 
\begin{align*}
\frac{1}{n^\beta} \left| \widehat{\derivative}(t) - \widehat{\derivative}(t') \right| & \le \left| I \left [ t, t + \frac{1}{n^\beta} \right ] - I \left [ t', t' + \frac{1}{n^\beta} \right] \right| + \left| I \left[ t - \frac{1}{n^\beta}, t \right] - I \left[ t' - \frac{1}{n^\beta}, t' \right ] \right| \\
& = \left| I \left[ t' + \frac{1}{n^\beta} , t + \frac{1}{n^\beta} \right] - I [ t', t] \right| + \left| I [ t', t] - I \left[ t' - \frac{1}{n^\beta} , t - \frac{1}{n^\beta} \right] \right| \\
& \le I \left[ t' + \frac{1}{n^\beta} , t + \frac{1}{n^\beta} \right] + I \left[ t' - \frac{1}{n^\beta} , t - \frac{1}{n^\beta} \right] + 2 I [ t', t ].
\end{align*}
Notice that the interval $[t' + n^{-\beta}, t + n^{-\beta}]$ intersects at most 2 intervals of the form $[a,b]$ for $a,b \in \mathbb{T}_n$, so we can bound $I [ t' + n^{-\beta}, t + n^{-\beta} ]$ by 2. The same argument holds for the intervals $[t'- n^{-\beta}, t - n^{-\beta} ]$ and $[t', t]$, leading to the bound $| \widehat{\derivative}(t) - \widehat{\derivative}(t') | \le 8n^\beta$. Identical arguments show that $| \widehat{\derivative}(t) - \widehat{\derivative}(t' + n^{-5}) | \le 8n^\beta$ as well.

Combining our results on $\derivative(t)$ and $\widehat{\derivative}(t)$, we see that 
\begin{equation}
\label{eq:derivative_difference_discretization}
\left| \widehat{\derivative}(t) - \derivative(t) \right| \le \max \left \{ \left| \widehat{\derivative}(t') - \derivative(t') \right| , \left | \widehat{\derivative}\left( t' + \frac{1}{n^5} \right) - \derivative \left( t' + \frac{1}{n^5} \right) \right| \right \} + 8n^\beta.
\end{equation}
We can therefore bound the probability of interest as follows. 
\begin{align*}
\p \left( \sup_{t \ge 0} \left| \widehat{\derivative}(t) - \derivative(t) \right| \mathbf{1}( \cB_t ) \ge n^\gamma \right) & \stackrel{(a)}{\le} \p \left( \max_{t \in \mathbb{T}_n} \left| \widehat{\derivative}(t) - \derivative(t) \right|\mathbf{1}( \cB_t ) \ge n^\gamma - 8n^\beta \right) \\
& \le \sum_{t \in \mathbb{T}_n} \p \left( \left \{ \left| \widehat{\derivative}(t) - \derivative(t) \right| \ge n^\gamma - 8n^\beta \right \} \cap \cB_t \right) \\
& \stackrel{(b)}{\le} n^7 \exp \left( - n^{\gamma + o(1)} \right) \\
& = o(1).
\end{align*}
Above, $(a)$ is a direct consequence of \eqref{eq:derivative_difference_discretization}. Inequality $(b)$ uses that $\gamma > \beta$, that there are $n^7$ elements of $\mathbb{T}_n$, and Lemma \ref{lemma:second_derivative_approximation}.
\end{proof}

\section{Concentration of the second derivative: Proof of Lemma \ref{lemma:second_derivative_approximation}}
\label{sec:second_derivative_proof}

The proof of Lemma \ref{lemma:second_derivative_approximation} relies on the following two results. 

\begin{lemma}
\label{lemma:cut_approximation}
Let $t \ge 0$, $\beta > 0$ and assume that $\gamma \in ((1 + \beta) / 2, 1)$. Then
\[
\p \left(  \left| I\left[ t, t + \frac{1}{n^\beta} \right] - \int_t^{t + n^{-\beta}} \hspace{-0.5cm} \lambda \cut ( \cascade(s)) ds \right| \ge n^{\gamma - \beta} \right) \le \exp \left( - n^{2 \gamma - (1 + \beta) + o(1)} \right).
\]
\end{lemma}

\begin{lemma}
\label{lemma:cut_changes}
Let $\gamma, \beta > 0$ satisfy $\gamma > 1 - \beta$. Recall the event 
\[
\cB_t : = \left \{ \forall v \in \highdeg(G), T(v) \notin \left( t - \frac{1}{n^\beta}, t \right) \cup \left( t , t + \frac{1}{n^\beta} \right) \right \}.
\]
Then the following hold:
\begin{align}
\label{eq:cut_diff_result_1}
\p \left( \left \{ \left| \int_t^{t + n^{-\beta}} \hspace{-0.5cm} \cut ( \cascade(s)) ds - \frac{ \cut ( \cascade(t))}{n^\beta} \right| \ge n^{\gamma - \beta}  \right \} \cap \cB_t \right) & \le \exp \left( - n^{\gamma + o(1)} \right), \\
\label{eq:cut_diff_result_2}
\p \left( \left \{ \left| \int_{t - n^{-\beta}}^{t} \hspace{-0.5cm} \cut ( \cascade(s)) ds - \frac{ \cut ( \cascade(t^-)) }{n^\beta} \right| \ge n^{\gamma - \beta}  \right \} \cap \cB_t \right) & \le \exp \left( - n^{\gamma + o(1)} \right).
\end{align}
\end{lemma}

Before proving the two lemmas, we show how they can be combined to prove the main result.

\begin{proof}[Proof of Lemma \ref{lemma:second_derivative_approximation}]
It is immediate from Lemmas \ref{lemma:cut_approximation} and \ref{lemma:cut_changes} and the triangle inequality that 
\begin{align}
\label{eq:cut_probability_upper}
\p \left( \left \{ \left| I \left[ t, t + \frac{1}{n^\beta} \right] - \frac{ \cut ( \cascade(t))}{n^\beta} \right| \ge 2n^{\gamma - \beta} \right \} \cap \cB_t \right) & \le 2 \exp \left( - n^{2\gamma - (1 + \beta) + o(1)} \right),  \\
\label{eq:cut_probability_lower}
\p \left( \left \{ \left| I \left[ t - \frac{1}{n^\beta}, t  \right] - \frac{ \cut ( \cascade(t^-))}{n^\beta} \right| \ge 2n^{\gamma - \beta} \right \} \cap \cB_t \right) & \le 2 \exp \left( - n^{2\gamma - (1 + \beta) + o(1)} \right).
\end{align}
Above, we have used the fact that $2 \gamma - (1 + \beta) < \gamma$ whenever $\gamma < 1$ to simplify the probability bounds. The bounds \eqref{eq:cut_probability_upper} and \eqref{eq:cut_probability_lower} along with the definitions of $\derivative(t)$ and $\hatderivative(t)$ (see \eqref{eq:expected_derivative} and \eqref{eq:empirical_derivative}) proves that
\[
\p \left ( \left \{ \left| \hatderivative(t) - \derivative(t) \right| \ge 4n^{\gamma} \right \} \cap \cB_t \right) \le 4 \exp \left( - n^{2 \gamma - (1 + \beta) + o(1)} \right),
\]
which is equivalent to the claimed result. 
\end{proof}

The rest of this section is devoted to the proofs of the two lemmas. The proof of Lemma \ref{lemma:cut_approximation} is in Section \ref{sec:cut_approximation_proof}, and the proof of Lemma \ref{lemma:cut_changes} is in Section \ref{sec:cut_changes_proof}.

\subsection{Approximating the cut: Proof of Lemma \ref{lemma:cut_approximation}}
\label{sec:cut_approximation_proof}

The key tool used to prove Lemma \ref{lemma:cut_approximation} is a version of Freedman's inequality due to Shorack and Wellner \cite[Appendix B]{shorack_wellner}.

\begin{lemma}[Freedman's inequality for continuous-time martingales]
\label{lemma:freedman}
Suppose that $\{ X(t) \}_{t \ge 0}$ is a real-valued continuous-time martingale adapted to the filtration $\{ \cF_t \}_{t \ge 0}$. Suppose further that $X_t$ has jumps that are bounded by $C > 0$ in absolute value, almost surely.
Define the quadratic variation of the process 
\[
\langle X \rangle_t : = \lim_{\xi \to 0} \sum_{k = 0}^{\lfloor t / \xi \rfloor} \E \left[ \left. \left( X({(k + 1) \xi}) - X({k \xi}) \right)^2 \right \vert \cF_{k\xi} \right].
\]
Then for any $x \ge 0$ and $\sigma > 0$, it holds that 
\[
\p ( X(t) - X(0) \ge x \text{ and } \langle X \rangle_t \le \sigma^2 \text{ for some $t \ge 0$} ) \le \exp \left( - \frac{x^2 / 2}{\sigma^2 + Cx / 3} \right).
\]
\end{lemma}

\begin{proof}
For $y \ge t$, define the process
\[
M(y) : = I(y) - I(t) - \int_t^y \lambda \cut( \cascade(s)) ds.
\]
Notice that $M(y)$ is a martingale. This is evident from the formula for the condition expectation of the derivative of $I$, given in \eqref{eq:J_definition}. To apply the concentration inequality of Lemma \ref{lemma:freedman}, we derive a few basic properties of $M(y)$. The process jumps by 1 precisely when a new vertex is infected, so we may set $C = 1$. To bound the quadratic variation, we have, for $\xi > 0$ sufficiently small, 
\begin{equation}
\label{eq:quadratic_variation_bound}
\E \left[ \left. ( M(y + \xi) - M(y) )^2 \right \vert \cF_y  \right] \le 2 \E \left[ \left. (I(y + \xi) - I(y) )^2  \right \vert \cF_y   \right] + 2 \E \left[ \left. \left( \int_y^{y + \xi} \lambda \cut( \cascade(s)) ds \right)^2 \right \vert \cF_y \right],
\end{equation}
where we have used the inequality $(a + b)^2 \le 2a^2 + 2b^2$. To bound the first term on the right hand side of \eqref{eq:quadratic_variation_bound}, we recall from Lemma \ref{lemma:infections_poisson_bound} that $I(y + \xi) - I(y)$ is stochastically bounded by a $\mathrm{Poi}( \lambda mnd \xi)$ random variable conditioned on $\cF_y$. It follows that
\begin{equation}
\label{eq:quadratic_variation_term_1}
\E \left[ \left. ( I(y + \xi) - I(y))^2 \right \vert \cF_y \right ] \le \lambda mnd \xi + (  \lambda mnd \xi)^2 =  \lambda mnd \xi + o ( \xi).
\end{equation}
To bound the second term on the right hand side of \eqref{eq:quadratic_variation_bound}, use the (somewhat loose) bound $\cut ( \cascade(s)) \le |E| \le n^2$ to arrive at
\begin{equation}
\label{eq:quadratic_variation_term_2}
\int_y^{y + \xi} \lambda \cut ( \cascade( s)) ds \le \xi \lambda n^2.
\end{equation}
Together, \eqref{eq:quadratic_variation_term_1} and \eqref{eq:quadratic_variation_term_2} show that
\[
\E \left[ \left. ( M(y + \xi) - M(y) )^2 \right \vert \cF_y  \right] \le 2 \lambda mnd \xi + o ( \xi),
\]
hence $\langle M \rangle_y \le 2 \lambda mnd (y - t)$.

Next, define the stopped process $\overline{M}(y) : = M(y \land (t + n^{-\beta}))$. For this process, we have that $M(t + n^{-\beta}) = \overline{M}(t + n^{-\beta})$ and that $\langle \overline{M} \rangle_y \le 2 \lambda md n^{1 - \beta}$ almost surely. Lemma \ref{lemma:freedman} applied to the process $\overline{M}$ shows that
\[
\p \left( {M}(t + n^{-\beta}) - {M}(t) \ge x \vert \cF_t \right) = 
\p \left( \overline{M}(t + n^{-\beta}) - \overline{M}(t) \ge x \vert \cF_t \right) \le \exp \left( - \frac{x^2 / 2}{2 \lambda md n^{1 - \beta} + x/3} \right).
\]
Applying the same analysis to the martingale $-M$ and taking a union bound shows that, for $x \ge 0$,
\[
\p \left( \left. \left| I\left[ t, t + \frac{1}{n^\beta} \right] - \int_t^{t + n^{-\beta}} \hspace{-0.5cm} \lambda \cut ( \cascade(s)) ds \right| \ge x \right \vert \cF_t \right) \le 2 \exp \left( - \frac{x^2 / 2}{2 \lambda md n^{1 - \beta} + x/3} \right).
\]
Substituting $x = n^{\gamma - \beta}$ into the expression above yields a probability bound of 
\[
 \exp \left( - n^{\min \{ \gamma - \beta, 2 \gamma - 1 - \beta \} + o(1)} \right).
\]
The desired claim follows since $2 \gamma - 1 - \beta < \gamma - \beta$ whenever $\gamma < 1$. 
\end{proof}

\subsection{Changes in the cut: Proof of Lemma \ref{lemma:cut_changes}}
\label{sec:cut_changes_proof}

As a shorthand, denote $\delta : = n^{-\beta}$. We can bound
\begin{align}
\label{eq:cut_integral}
\left| \int_t^{t + \delta} \cut ( \cascade(s)) ds - \delta \cut ( \cascade(t)) \right| & \le \int_t^{t + \delta} | \cut ( \cascade(s)) - \cut ( \cascade(t)) | ds  \le \delta \sup_{t < s < t + \delta} | \cut ( \cascade(s)) - \cut (\cascade(t)) |.
\end{align}
Recall the event $\cB_t : = \{ \forall v \in \highdeg(G), T(v) \notin (t - \delta, t) \cup (t, t + \delta) \}$. 
On $\cB_t$, all vertices in $\cascade(t + \delta) \setminus \cascade(t)$ have degree at most $d$. As a result, for all $s \in (t, t + \delta)$, it holds that
\begin{equation}
\label{eq:cut_difference}
\left| \cut ( \cascade(s) ) - \cut( \cascade(t)) \right| \le \sum_{v \in \cascade(s) \setminus \cascade(t)} \deg(v) \le d | \cascade(s) \setminus \cascade(t) | \le d | \cascade(t + \delta) \setminus \cascade(t) |.
\end{equation}
The first inequality in \eqref{eq:cut_difference} follows since the size of the cut changes by at most $\deg(v)$ vertices when $v$ is added to the cascade. The second inequality holds since $\cascade(t) \subseteq \cascade(s) \subseteq \cascade(t + \delta)$.

We proceed by studying $| \cascade(t + \delta) \setminus \cascade(t) |$. Let $X \sim \mathrm{Poi}(\lambda mn d \delta )$. By Lemma \ref{lemma:infections_poisson_bound}, it holds that
\begin{align}
\p \left( | \cascade(t + \delta) \setminus \cascade(t) | \ge \frac{n^\gamma}{d} \right) & \stackrel{(a)}{\le} \p \left ( X \ge \frac{n^\gamma}{2d} + \E [ X] \right) \nonumber \\
& \stackrel{(b)}{\le} \mathrm{exp} \left( - \frac{n^{2\gamma}/(8d^2)}{ n^{1 - \beta + o(1)} + n^\gamma / (6d)} \right) \nonumber \\
\label{eq:cut_diff_bernstein}
& \stackrel{(c)}{\le} \mathrm{exp} \left( - n^{\gamma + o(1)} \right).
\end{align}
Above, $(a)$ holds since $\E [ X ] \le \lambda m n^{1 - \beta + o(1)} \le n^\gamma /(2d)$ provided $\gamma > 1 - \beta$. The inequality $(b)$ follows from Bernstein's inequality and the fact that $\lambda mnd = n^{1 + o(1)}$. The final inequality $(c)$ holds by our assumption that $\gamma > 1 - \beta$. Together, \eqref{eq:cut_integral}, \eqref{eq:cut_difference} and \eqref{eq:cut_diff_bernstein} imply the result \eqref{eq:cut_diff_result_1}.

To prove \eqref{eq:cut_diff_result_1}, we may follow identical reasoning. Analogously to \eqref{eq:cut_integral} and \eqref{eq:cut_difference}, it holds on the event $\cB_t$ that
\[
\left| \int_{t - \delta}^t \cut ( \cascade(s)) ds - \delta \cut( \cascade(t^-)) \right| \le \delta \sup_{t - \delta < s< t} | \cut (\cascade(s)) - \cut (\cascade(t)) | \le d | \cascade(t) \setminus \cascade(t - \delta) |.
\]
We may then analyze $| \cascade(t) \setminus \cascade(t - \delta) |$ in the same manner as \eqref{eq:cut_diff_bernstein}. Putting everything together yields the second result \eqref{eq:cut_diff_result_2}.%
\hfill \qed

\section{Lower bounds for $\alpha \in (0,1/2)$: Proof of Theorem \ref{thm:impossibility}}
\label{sec:impossibility_proof}

To prove the theorem, we will construct a tailored family of graphs $\cH \subset \cG$ for which we can tractably analyze the distribution of infection times, while also constituting a ``hard'' instance for estimating high-degree vertices. To construct $\cH$, we first describe a particular deterministic graph $H$.

\begin{definition}
\label{def:H}
Fix a positive integer $N$. Let $H'$ be a connected tree with $N$ vertices, maximum degree at most 3, and diameter at most $2 \log N$ (e.g., a balanced binary tree). For each vertex in $u \in H'$, we add a path of length $\log^5 n$ starting at $u$; all added paths are disjoint. We call the resulting graph $H$. The set of endpoints of the paths added to $H'$, which are also the set of leaves of $H$, is denoted by $L$. 
\end{definition}

The addition of paths which are much longer than the diameter of $H'$ is a crucial aspect of our construction. At a high level, it ensures that the infection times of vertices in $L$ are primarily determined by the infection times of vertices in the corresponding path, rather than the infection times of vertices in $H'$. Since all the added paths are disjoint, the infection times of vertices in $L$ become \emph{almost} independent, which simplifies much of our analysis. A bit more formally, we have the following useful observation.

\begin{fact}
\label{fact:conditional_independence_H}
The infection times of vertices in $L$, given by the collection $\{ \mathbf{T}(v) \}_{v \in L}$, are conditionally independent with respect to $\{ \mathbf{T}(z) \}_{z \in H'}$, the collection of infection times of vertices in $H'$.
\end{fact}


To construct the graphs of interest, we will introduce additional vertices to $H$ which connect to $L$ in a particular manner. 

\begin{definition}
Suppose that $N$ vertices, labelled from 1 to $N$, are added to $H$. Each vertex in $[N]$ forms an edge with a single element of $L$. Define $\cH_\emptyset$ to be the set of resulting graphs for which all vertices in $L$ have degree at most $\log^2 n$. For each $v \in L$, define $\cH_v$ to be the set of graphs for which each $u \in L \setminus \{v \}$ has degree at most $\log^2 n$, but $v$ has degree at least $D / 2$. The union of all such graphs is given by $\cH : = \bigcup_{x \in L \cup \{ \emptyset \}} \cH_x$.
\end{definition}

We make a few simple observations about the family $\cH$. 
\begin{itemize}

\item If $G \in \cH_\emptyset$ then $\highdeg(G) = \emptyset$, and if $G \in \cH_v$ for $v \in L$ then $\highdeg(G) = \{v\}$.

\item The total number of vertices in $G \in \cH$ is $n : = 2N + N \log^5 N$. In particular, $N = n^{1 - o(1)}$. 

\end{itemize}

The first step in our analysis is to show that it is challenging to even detect whether the set of observed infection times are due to a graph in $\cH_\emptyset$ or in $\cH_v$ for some $v \in L$. To do so, we construct convenient distributions $\mu_\emptyset$ over elements of $\cH_\emptyset$ and $\mu_v$ over elements of $\cH_v, v \in L$. We let $P_\emptyset$ and $P_v, v \in L$ denote the induced distributions over the infection times. That is, for $x \in L \cup \{ \emptyset \}$, $P_x$ is the distribution corresponding to $\SI(G, v_0)^{\otimes K}$ where $G \sim \mu_x$ and $v_0 \i H'$ is chosen arbitrarily (but is fixed across all cascades). 
We defer the precise description of the distributions $\{ \mu_x \}_{x \in L \cup \{ \emptyset \}}$ to Section \ref{sec:graph_distributions}, and move forward with the proof for now.

\begin{lemma}
\label{lemma:detection_impossibility}
Let $\epsilon > 0$ and $\alpha \in (0,1/2)$. Suppose that $D \le n^\alpha$ for some $\alpha \in (0,1/2)$ and that 
\[
K \le \left( \frac{ 1 - 2 \alpha - \epsilon}{5} \right) \frac{\log n}{\log \log n}.
\]
Then it holds for all $v \in L$ that $\TV( P_\emptyset, P_v) = o(1)$.  
\end{lemma}

At a high level, we prove Lemma \ref{lemma:detection_impossibility} by studying the $\chi^2$-divergence between $P_\emptyset$ and $P_v$ for $v \in L$, which provides an upper bound for the total variation distance and has a useful tensorization property which we leverage in our analysis. 
Before proving the lemma, we show how it can be used to prove the theorem.


\begin{proof}[Proof of Theorem \ref{thm:impossibility}]
Define the measure $\mu$ over the ensemble $\cH$ via $\mu : = \frac{1}{N+1} \sum_{x \in L \cup \{ \emptyset \}} \mu_x$, and similarly define the induced measure over vertex infection times $P : = \frac{1}{N + 1} \sum_{x \in L \cup \{ \emptyset \}} P_x$. Let us also define $\delta : = \max_{x,y \in L \cup \{ \emptyset \}} \TV(P_x, P_y)$ and note that $\delta = o(1)$ by Lemma \ref{lemma:detection_impossibility} and the triangle inequality. For any estimator $\mathrm{HD}$, it holds that
\begin{align*}
P ( \mathrm{HD} = \highdeg(G) ) & = \frac{1}{N + 1} \sum_{x \in L \cup \{ \emptyset \}} P_x ( \mathrm{HD} = \{ x \} ) \\
& \stackrel{(a)}{=} 1 - \frac{1}{N + 1} \sum_{x,y \in L \cup \{ \emptyset \}: x \neq y} P_x ( \mathrm{HD} = \{ y \} ) \\
& \stackrel{(b)}{\le} 1 - \frac{1}{N + 1} \sum_{x,y \in L \cup \{ \emptyset \} : x \neq y} \left( P_y ( \mathrm{HD} = \{ y \} ) - \delta \right) \\
& \stackrel{(c)}{=} 1 + N \delta - N P ( \mathrm{HD} = \highdeg(G) ),
\end{align*}
where $(a)$ is due to the law of total probability, $(b)$ uses the definition of the total variation distance, and $(c)$ follows as each term $P_y ( \mathrm{HD} = \{y \})$ is counted $N$ times in the summation. Rearranging terms shows that
\[
P ( \mathrm{HD} = \highdeg(G) ) \le \frac{1 + N \delta}{1 + N} = o(1).
\]
\end{proof}

The remainder of this section is devoted to the proof of Lemma \ref{lemma:detection_impossibility}.

\subsection{Distributions of graphs in $\cH$}
\label{sec:graph_distributions}

Before proving the lemma, we must describe the distributions $\{ \mu_x \}_{x \in L \cup \{ \emptyset \}}$ over graphs in $\cH$. To this end, it is useful to first define corresponding ``auxiliary'' distributions $\{ \mu_x' \}_{x \in L \cup \{\emptyset \}}$. We say that $G \sim \mu_\emptyset'$ if each vertex in $[N]$ chooses a uniform random element of $L$ with which they form an edge. On the other hand, we say that $G \sim \mu_v'$ if, for each $u \in [N]$, $u$ connects to a uniform random element of $L$ with probability $1 - D / N$, otherwise $u$ connects to $v$ with probability $D / N$. Finally, for $x \in L \cup \{ \emptyset \}$, the measure $\mu_x$ is equal to $\mu_x'$ conditioned on the sampled graph being an element of $\cH_x$. 

For $x \in L \cup \{ \emptyset \}$, we recall that $P_x$ is the distribution corresponding to $\SI(G, v_0)^{\otimes K}$, where $G \sim \mu_x$. Similarly, we define $P_x'$ to be the distribution corresponding to $\SI(G, v_0)^{\otimes K}$, where $G \sim \mu_x'$. The following result shows that the total variation distance between $P_x$ and $P_x'$ is small. 

\begin{lemma}
\label{lemma:tv_auxiliary}
It holds for any $x \in L \cup \{ \emptyset \}$ that $\TV( P_x, P_x') = o(1)$. 
\end{lemma}

The lemma implies that we can essentially replace $P_x$ with $P_x'$ in our analysis moving forward. Indeed, by the triangle inequality for the total variation distance, we have for any $v \in L$ that
\begin{equation}
\label{eq:tv_auxiliary}
\TV( P_\emptyset, P_v) \le \TV( P_\emptyset', P_v') + \TV( P_\emptyset', P_\emptyset) + \TV( P_v', P_v) = \TV ( P_\emptyset', P_v') + o(1),
\end{equation}
where the final equality is due to Lemma \ref{lemma:tv_auxiliary}. Hence it suffices to bound the total variation between the auxiliary distributions. 

\begin{proof}[Proof of Lemma \ref{lemma:tv_auxiliary}]
Let $x \in L \cup \{ \emptyset \}$. By the data processing inequality for the total variation distance, it holds that $\TV(P_x, P_x') \le \TV( \mu_x, \mu_x')$. By the definition of $\mu_x$ and $\mu_x'$, the two distributions can be coupled precisely when $G \sim \mu_x'$ satisfies $G \in \cH_x$. Hence 
\begin{equation}
\label{eq:tv_auxiliary_bound}
\TV(P_x, P_x') \le \TV(\mu_x, \mu_x') = \mu_x'(\cH_x^c).
\end{equation}
In light of \eqref{eq:tv_auxiliary_bound}, we proceed by proving that $\mu_x'( \cH_x^c) = o(1)$. Notice that for $G \sim \mu_x'$, all vertices except for those in $L$ have a deterministic degree which is at most 3, hence it suffices to study the degrees of vertices in $L$ alone. In the case $x = \emptyset$, it holds for every $v \in L$ that $\deg_G(v) - 1 \sim \Bin(N, 1/N)$. Bernstein's inequality implies
\begin{equation}
\label{eq:degree_bound}
\mu_\emptyset' \left( \deg_G(v) \ge \log^2 N \right) \le \mu_\emptyset' \left( \deg_G(v) - 1 \ge 1 + \frac{1}{2} \log^2 N \right) \le \mathrm{exp} \left( - \frac{2}{3} \log^2 N \right).
\end{equation}
It follows that
\begin{align}
\mu_\emptyset' \left( \cH_\emptyset^c \right) & = \mu_\emptyset' \left( \exists v \in L : \deg_G(v) > \log^2 n \right) \nonumber \\
& \stackrel{(a)}{\le} \mu_\emptyset' \left( \exists v \in L : \deg_G(v) > \log^2 N \right) \nonumber \\
\label{eq:auxiliary_prob_bound_emptyset}
& \stackrel{(b)}{\le} N \exp \left( - \frac{2}{3} \log^2 N \right) = o(1).
\end{align}
Above, $(a)$ follows since $N \le n$, and $(b)$ is due to a union bound over the $N$ elements of $L$. 

We now turn to the case where $x \in L$. If $v \in L \setminus \{x \}$, then $\deg_G(v) - 1 \sim \Bin \left( N , \left( 1 - \frac{D}{N} \right) \frac{1}{N} \right)$, which is stochastically bounded by $\Bin( N , 1/N)$. Following the same steps as \eqref{eq:degree_bound}, we have that
\[
\mu_x' \left( \deg_G(v) \ge \log^2 n \right) \le \mu_x' \left( \deg_G(v) \ge \log^2 N \right) \le \exp \left( - \frac{2}{3} \log^2 N \right).
\]
On the other hand, when $v = x$, $\deg_G(v)$ stochastically dominates $\Bin( N, D / N)$. By Bernstein's inequality, 
\[
\mu_x' \left( \deg_G(v) \le \frac{D}{2} \right) \le \exp \left( - \frac{3}{56} D \right).
\]
It follows that 
\begin{align}
\mu_x' \left( \cH_x^c \right) & = \mu_x' \left( \exists v \in L \setminus \{x \} : \deg_G (v) > \log^2 n \text{ or } \deg_G(x) < \frac{D}{2} \right) \nonumber \\
& \le \sum_{v \in L \setminus \{x \} } \mu_x' \left( \deg_G(v) > \log^2 n \right) + \mu_x' \left( \deg_G(x) < \frac{D}{2} \right) \nonumber \\
\label{eq:auxiliary_prob_bound}
& \le N \exp \left ( - \frac{2}{3} \log^2 N \right) + \exp \left( - \frac{3}{56} D \right) = o(1).
\end{align}
Together, \eqref{eq:auxiliary_prob_bound_emptyset}, \eqref{eq:auxiliary_prob_bound} and \eqref{eq:tv_auxiliary_bound} prove the claim.
\end{proof}

\subsection{Impossibility of detection: Proof of Lemma \ref{lemma:detection_impossibility}}

The advantage of using the auxiliary distributions is that they are product distributions conditioned on the infection times in $H$, which makes it easier to analyze the infection times in $[N]$. To see this concretely, let us assume without loss of generality that $\lambda = 1$, and define
\[
\mathbf{E}(j) = (E_1(j), \ldots, E_K(j) ) \stackrel{i.i.d.}{\sim} \mathrm{Exp}(1)^{\otimes K}, \qquad j \in [N].
\]
Let us also define the collection of independent variables $\{ U(j) \}_{j \in [N]}$ where $U(j) \sim \mathrm{Unif}(L)$ denotes the random neighbor of the vertex $j$ in the construction of $G \sim P_\emptyset'$. We then have the representation
\begin{equation}
\label{eq:infection_times_P0}
\mathbf{T}(j) = \mathbf{T}( U(j)) + \mathbf{E}(j), \qquad j \in [N].
\end{equation}
On the other hand, if $G \sim P_v'$ for some $v \in L$, then we have the representation
\begin{equation}
\label{eq:infection_times_Pv}
\mathbf{T}(j) = \begin{cases}
\mathbf{T}( U(j) ) + \mathbf{E}(j) & \text{ with probability $1 - D / N$} \\
\mathbf{T}(v) + \mathbf{E}(j) & \text{ with probability $D / N$.}
\end{cases}
\end{equation}

We can also succinctly characterize the distributions of infection times in $[N]$ for the cases described in \eqref{eq:infection_times_P0} and \eqref{eq:infection_times_Pv}. Define the following (conditional) distributions over $\mathbf{t} : = (t_1, \ldots, t_K) \in \R_+^K$,
\begin{align*}
f_v(\mathbf{t}) & : = \prod_{i = 1}^K e^{ - (t_i - T_i(v))} \mathbf{1}( t_i \ge T_i(v) ) = e^{- \sum_{i = 1}^K t_i } \prod_{i = 1}^K e^{T_i(v) } \mathbf{1} ( t_i \ge T_i(v) ) \\
f_\emptyset ( \mathbf{t}) & : = \frac{1}{N} \sum_{w \in L} f_w(\mathbf{t}) = \frac{e^{ - \sum_{i = 1}^K t_i } }{N} \sum_{w \in L} \prod_{i = 1}^K e^{T_i(w)} \mathbf{1} ( t_i \ge T_i(w) ).
\end{align*}
Then it is readily seen that 
\begin{equation}
\label{eq:conditional_infection_times}
\mathrm{Law} \left( \left. \{ \mathbf{T}(j) \}_{j \in [N]} \right| \{ \mathbf{T}(u) \}_{u \in L} \right) = \begin{cases}
f_\emptyset^{\otimes N} & G \sim P_\emptyset' \\
\left[ \left( 1 - \frac{D}{N} \right) f_\emptyset + \frac{D}{N} f_v \right]^{\otimes N} & G \sim P_v', v \in L.
\end{cases}
\end{equation}
Our strategy moving forward is to determine when the distributions in \eqref{eq:conditional_infection_times} are indistinguishable, which amounts to bounding the total variation distance between them. To do so, we will instead bound a larger distance measure -- the $\chi^2$ divergence. We recall that the $\chi^2$ divergence between probability measures $P$ and $Q$ is defined as 
\[
\chi^2 (P \| Q) : = \int \frac{( dP - dQ)^2 }{dQ} .
\]
The total variation distance and the $\chi^2$ divergence can be related through the inequality \cite[Proposition 7.15]{polyanskiy_wu_information_theory}
\[
\TV( P, Q) \le 2 \sqrt{ \chi^2 ( P \| Q )}.
\]
For our setting, the $\chi^2$ divergence has two advantageous properties. First, it tensorizes; that is, for any positive integer $k$, it holds that \cite[Chapter 7.12]{polyanskiy_wu_information_theory}
\[
\chi^2 \left( \left.  P^{\otimes k} \right \| Q^{\otimes k} \right) = \left( 1 + \chi^2 ( P \| Q ) \right)^k - 1.
\]
Additionally, the formula for the $\chi^2$ divergence involves the difference between the densities of $P$ and $Q$, which makes it convenient to analyze mixture distributions, as is the case for our setting.

We apply these ideas to the conditional distributions in \eqref{eq:conditional_infection_times}. Conditioned on $\{ \mathbf{T}(u) \}_{u \in H}$, we have that
\begin{align*}
 \TV \left( f_\emptyset^{\otimes N}, \left[ \left( 1 - \frac{D}{N} \right) f_\emptyset + \frac{D}{N} f_v \right]^{\otimes N} \right) & \le 2 \sqrt{  \chi^2 \left(\left. \left[   \left( 1 - \frac{D}{N} \right) f_\emptyset + \frac{D}{N} f_v \right]^{\otimes N} \right \| f_\emptyset^{\otimes N} \right) } \\
 & = 2 \sqrt{  \left[ 1 + \chi^2 \left(\left.  \left( 1 - \frac{D}{N} \right) f_\emptyset + \frac{D}{N} f_v \right \| f_\emptyset \right) \right]^N - 1 } \\
 & \le 2 \sqrt{   \exp \left \{ N \chi^2 \left(\left.  \left( 1 - \frac{D}{N} \right) f_\emptyset + \frac{D}{N} f_v \right \| f_\emptyset \right)  \right \} - 1 }.
\end{align*}
In particular, we see that the conditional total variation tends to 0 if 
\begin{equation}
\label{eq:chi_squared_condition}
\chi^2 \left(\left.  \left( 1 - \frac{D}{N} \right) f_\emptyset + \frac{D}{N} f_v \right \| f_\emptyset \right) = o \left( \frac{1}{N} \right).
\end{equation}
We proceed by bounding the $\chi^2$ divergence between the two distributions. Letting $\mathbf{t} = ( t_1, \ldots, t_K) \in \R^K$, it holds that
\begin{equation}
\label{eq:chi_squared_bound}
\chi^2 \left(\left.  \left( 1 - \frac{D}{N} \right) f_\emptyset + \frac{D}{N} f_v \right \| f_\emptyset \right)  = \left( \frac{D}{N} \right)^2 \int_{\R_+^K} \frac{ (f_v( \mathbf{t}) - f_\emptyset ( \mathbf{t}))^2 }{f_\emptyset ( \mathbf{t})}  d\mathbf{t} \le \left( \frac{D}{N} \right)^2 \int_{\R_+^K} \left( \frac{f_v ( \mathbf{t}) }{f_\emptyset( \mathbf{t})} \right) f_v( \mathbf{t})  d\mathbf{t}.  
\end{equation}
To bound the $\chi^2$ divergence, we therefore need to study the likelihood ratio $f_v / f_\emptyset$. This is accomplished in the following result. 

\begin{lemma}
\label{lemma:likelihood_ratio}
For every $v \in L$, it holds with probability $1 - o(1)$ that
\[
\sup_{\mathbf{t} \in \R_+^K} \frac{ f_v(\mathbf{t}) }{ f_\emptyset( \mathbf{t} ) } \le \log^{5K} n.
\]
\end{lemma}

Before proving Lemma \ref{lemma:likelihood_ratio}, we show how it can be used to prove the main result.

\begin{proof}[Proof of Lemma \ref{lemma:detection_impossibility}]
Let $\cA_v$ denote the event described in Lemma \ref{lemma:likelihood_ratio}; note that $\cA_v$ is measurable with respect to $\{ \mathbf{T}(u) \}_{u \in L}$. On this event, it follows from \eqref{eq:chi_squared_bound} that
\[
\chi^2 \left(\left.  \left( 1 - \frac{D}{N} \right) f_\emptyset + \frac{D}{N} f_v \right \| f_\emptyset \right) \le \left( \frac{D}{N} \right)^2 \int_{\R_+^K} \left( \log^{5K} n \right) f_v ( \mathbf{t} ) d \mathbf{t} \le \left( \frac{D}{N} \right)^2 \log^{5K} n. 
\]
Since $D/N = n^{\alpha - 1 + o(1)}$, for any $\epsilon > 0$ the condition \eqref{eq:chi_squared_condition} is satisfied if
\begin{equation}
\label{eq:TV_K_bound}
K \le \frac{1 - 2 \alpha - \epsilon}{5} \left( \frac{\log n}{ \log \log n} \right).
\end{equation}
We now translate our results on the conditional distributions $f_\emptyset^{\otimes N}$ and $[ (1 - D / N) f_\emptyset + (D/N) f_v ]^{\otimes N}$ to their \emph{unconditional} counterparts $P_\emptyset'$ and $P_v'$. If \eqref{eq:TV_K_bound} holds, then
\begin{align*}
\TV( P_\emptyset', P_v') & \stackrel{(a)}{\le} \E \left[ \TV \left( f_\emptyset^{\otimes N}, \left[ \left( 1 - \frac{D}{N} \right) f_\emptyset + \frac{D}{N} f_v \right]^{\otimes N} \right) \right] \\
& \stackrel{(b)}{\le} \E \left[ \TV \left( f_\emptyset^{\otimes N}, \left[ \left( 1 - \frac{D}{N} \right) f_\emptyset + \frac{D}{N} f_v \right]^{\otimes N} \right) \mathbf{1}( \cA_v ) \right] + \p ( \cA_v^c) = o(1).
\end{align*}
Above, the inequality $(a)$ follows from the property that conditioning increases the total variation \cite[Theorem 7.5]{polyanskiy_wu_information_theory}, and inequality $(b)$ holds since the total variation is always upper bounded by 1. Finally, $\TV( P_\emptyset', P_v') = o(1)$ implies that $\TV(P_\emptyset, P_v) = o(1)$ as well in light of \eqref{eq:tv_auxiliary}. 
\end{proof}

In the remainder of this subsection, we study the likelihood ratio
\[
\frac{f_v( \mathbf{t}) }{f_\emptyset( \mathbf{t} ) } = \frac{ \prod_{i = 1}^K e^{T_i(v) } \mathbf{1}( t_i \ge T_i(v) ) }{\frac{1}{N} \sum_{w \in L} \prod_{i = 1}^K e^{T_i(w)} \mathbf{1}( t_i \ge T_i(w))}.
\]
When there exists $i \in [K]$ such that $t_i < T_i(v)$, it is clear from the formula that the likelihood ratio is zero. On the other hand, when $t_i \ge T_i(v)$ for all $i \in [K]$, 
\begin{align*}
\frac{f_v( \mathbf{t}) }{f_\emptyset( \mathbf{t} ) } & = \frac{ \prod_{i = 1}^K e^{T_i(v)} }{ \frac{1}{N} \sum_{w \in L} \prod_{i = 1}^K e^{T_i(w)} \mathbf{1}(t_i \ge T_i(w)) } \\
& \le \frac{ \prod_{i = 1}^K e^{T_i(v) } }{ \frac{1}{N} \sum_{w \in L} \prod_{i = 1}^K e^{T_i(w)} \mathbf{1}( T_i(v) \ge T_i(w)) } \\
& = \left( \frac{1}{N} \sum_{w \in L} \prod_{i = 1}^K e^{T_i(w) - T_i(v)} \mathbf{1}( T_i(v) \ge T_i(w) ) \right)^{-1}.
\end{align*}
Above, the inequality on the second line holds since $\mathbf{1}(T_i(v) \ge T_i(w)) \le \mathbf{1}( t_i \ge T_i(w))$ if $t_i \ge T_i(v)$. Importantly, the upper bound derived for the likelihood ratio does not depend on $\mathbf{t}$. Hence, to control $\sup_{\mathbf{t} \in \R_+^K} f_v(\mathbf{t}) / f_\emptyset( \mathbf{t})$, we establish a lower bound for the quantity
\begin{equation}
\label{eq:likelihood_inverse}
\frac{1}{N} \sum_{w \in L} \prod_{i = 1}^K e^{T_i(w) - T_i(v)} \mathbf{1}( T_i(v) \ge T_i(w) ) .
\end{equation}

We remark that the analysis of \eqref{eq:likelihood_inverse} is greatly facilitated by the way in which the graph $H$ is constructed (see Definition \ref{def:H}). To see this concretely, for each $v \in L$ let $v'$ be the first vertex in $H'$ on the unique path in $H$ connecting $v$ to $H'$. Then we can write
\begin{equation}
\label{eq:T_X_distribution}
T_i(v) = T_i(v') + X_i(v), \qquad i \in [K],
\end{equation}
where $X_i(v)$ is the time it takes for the cascade to spread from $v'$ to $v$. Importantly, as the paths added to $H'$ to form $H$ are all disjoint, $\{ X_i(v) \}_{v \in L}$ is a collection of i.i.d. $\mathrm{Gamma}( \log^5 N)$ random variables. 
Since the diameter of $H'$ is assumed to be much smaller than $\log^5 N$, we can show that $T_i(v') = o (\log^5 N)$, and the behavior of $T_i(v)$ is largely dictated by $X_i(v)$, a random variable that is independent across $v \in L$. As we shall see, this reduction to i.i.d. random variables leads to a tractable analysis of \eqref{eq:likelihood_inverse}. 

To formalize these ideas, we begin by defining the following ``nice'' event. 

\begin{definition}
\label{def:E}
Let $v \in L$. The event $\cE_v$ holds if and only if the following conditions are satisfied:
\begin{enumerate}
\item $T_i(z) \le \log^2 N$ for all $z \in H'$ and $i \in [K]$.
\item  $T_i(v) \ge \frac{2}{3} \log^5 N$ for all $i \in [K]$.
\item $ \sum_{i = 1}^K \left ( X_i(v) - \log^5 N \right )^2 \le 2K \log^5 N$.
\end{enumerate}
\end{definition}

Importantly, as the following result shows, $\cE_v$ holds with high probability. 

\begin{lemma}
For all $v \in L$, it holds that $\p( \cE_v) = 1 - o(1)$. 
\end{lemma}

\begin{proof}
We let $\cE_{v,1}, \cE_{v,2}$ and $\cE_{v,3}$ denote the events corresponding to the first, second and third conditions of $\cE_v$ in Definition \ref{def:E}. 
We start by showing that $\cE_{v,1}$ holds with probability $1 - o(1)$. 
Since $H'$ is a tree and the diameter of $H'$ is at most $2 \log N$ (see Definition \ref{def:H}), we have that $T_i(z)$ is stochastically upper bounded by a $\mathrm{Gamma}(2 \log N, 1)$ random variable. Hence
\[
 \p \left( T_i(z) \ge \log^2 N \right) \le e^{- \frac{1}{2} \log^2 N} 2^{\log N} \le e^{- \frac{1}{4} \log^2 N},
\]
where the first inequality is due to a Chernoff bound, and the second inequality holds for $N$ sufficiently large. Taking a union bound over all $z \in H'$ and $i \in [K]$, we obtain 
\begin{equation}
\label{eq:E1_bound}
\p ( \cE_{v,1}^c) = \p \left( \max_{z \in H', i \in [K]} T_i(z) \ge \log^2 N \right) \le \sum_{z \in H', i \in [K]} \p \left( T_i(z) \ge \log^2 N \right) \le N K e^{- \frac{1}{4} \log^2 N} = o(1).
\end{equation}
We now turn to the events $\cE_{v,2}$ and $\cE_{v,3}$. Notice that if $\cE_{v,1}$ holds and there exists $i \in [K]$ such that $T_i(v) < \frac{2}{3} \log^5 N$, then 
\begin{equation}
\label{eq:E2_implies_E3}
\sum_{i = 1}^K \left( X_i(v) - \log^5 N \right)^2 \ge \left( \frac{\log^5 N}{4} \right)^2 > 2K \log^5 N,
\end{equation}
where the final inequality holds for $N$ sufficiently large. As a consequence of \eqref{eq:E2_implies_E3}, we have that 
\begin{equation}
\label{eq:E2_bound}
\cE_{v,1} \cap \cE_{v,2}^c \subseteq \cE_{v,3}^c,
\end{equation}
so we focus on bounding the probability of the latter event. We do so through a second moment argument. As a shorthand, let $R : = \log^5 N$. 
The first moment of each of the summands is $\E [ (X_i(v) - R)^2 ] = \mathrm{Var}(X_i(v)) = R$, where we have used the fact that $X_i(v) \sim \mathrm{Gamma}(R,1)$. The second moment is
\begin{align*}
\E \left[ (X_i(v) - R)^4 \right] & = \E [ X_i(v)^4 ] - 4 R \E [ X_i(v)^3 ] + 6 R^2 \E [ X_i(v)^2 ] - 4 R^3 \E [ X_i(v) ] + R^4 \\
& = R ( R + 1) ( R + 2) ( R + 3) - 4R^2 (R + 1) (R + 2) + 6 R^3 (R + 1) - 4 R^4 + R^4 \\
& = 3R^2 + 6R.
\end{align*}
In particular, $\mathrm{Var} ( (X_i(v) - R)^2 ) = 2R^2 + 6R \le 8R^2$, and furthermore, 
\[
\Var \left (  \sum_{i = 1}^K ( X_i(v) - R)^2 \right) =  \sum_{i = 1}^K \Var ( (X_i(v)-  R)^2 ) \le 8K R^2.
\]
Finally, by Chebyshev's inequality, 
\begin{equation}
\label{eq:E3_bound}
\p ( \cE_{v,3}^c ) = \p \left(  \sum_{i = 1}^K ( X_i(v) - R)^2 > 2KR \right) \le \frac{1}{K^2 R^2} \Var \left (\sum_{i = 1}^K ( X_i(v) - R)^2 \right) \le \frac{8}{K} = o(1). 
\end{equation}
Together, \eqref{eq:E1_bound}, \eqref{eq:E2_bound} and \eqref{eq:E3_bound} prove that $\p ( \cE_v^c) = o(1)$ as claimed.
\end{proof}

Next, we state a simple consequence of Taylor's theorem that will be useful for the proof of Lemma \ref{lemma:likelihood_ratio}.

\begin{lemma}
\label{lemma:exponential_inequality}
Suppose that $a,b$ are real numbers such that $b \neq 0$ and $a / b \ge - 1/2$. Then 
\[
e^{-a} \left( 1 + \frac{a}{b} \right)^b \ge e^{ - 2a^2 / b}.
\]
\end{lemma}

\begin{proof}
Taylor's theorem implies that for any $z \ge -1/2$, it holds that $\log(1 + z) \ge z - 2 z^2$. Using this inequality, we have that
\begin{align*}
\log \left( e^{-a} \left( 1 + \frac{a}{b} \right)^b \right) & = - a + b \log \left( 1 + \frac{a}{b} \right)  \ge - a + b \left( \frac{a}{b} - 2 \left( \frac{a}{b} \right)^2 \right) = - \frac{2a^2}{b}.
\end{align*}
\end{proof}

We now prove the main result of this section.

\begin{proof}[Proof of Lemma \ref{lemma:likelihood_ratio}]
Throughout this proof, we shall condition on the infection times in $H'$, given by the collection $\{ \mathbf{T}(z) \}_{z \in H'}$, as well as $v$, given by $\mathbf{T}(v)$. We will also assume that the ``nice'' event $\cE_v$ holds, which is measurable with respect to $\{ \mathbf{T}(z) \}_{z \in H'}, \mathbf{T}(v)$. Finally, as a shorthand, we will denote $R := \log^5 N$.

Recall from Fact \ref{fact:conditional_independence_H} that $\{ \mathbf{T}(u) \}_{u \in L}$ is a collection of conditionally independent variables with respect to $\{ \mathbf{T}(z) \}_{z \in H'}$. As a result, if we condition on both $\{ \mathbf{T}(z) \}_{z \in H'}$ and $\mathbf{T}(v)$, then the terms
\begin{equation}
\label{eq:B_product_terms}
\prod_{i = 1}^K e^{T_i(w) - T_i(v)} \mathbf{1}( T_i(v) \ge T_i(w) ), \qquad w \in L,
\end{equation}
are mutually (conditionally) independent. Moreover, each of the terms in \eqref{eq:B_product_terms} are bounded between 0 and 1. Hoeffding's inequality therefore implies that
\begin{multline}
 \frac{1}{N} \sum_{w \in L} \prod_{i = 1}^K e^{T_i(w) - T_i(v)} \mathbf{1} ( T_i(v) \ge T_i(w))  \\
 \label{eq:B_conditional_expectation}
 \ge \E \left[ \left. \frac{1}{N} \sum_{w \in L} \prod_{i = 1}^K e^{T_i(w) - T_i(v)} \mathbf{1} ( T_i(v) \ge T_i(w)) \right \vert \{ \mathbf{T}(z) \}_{z \in H'}, \mathbf{T}(v) \right] - N^{-1/3},
\end{multline}
with probability at least $1 - 2 \exp( - 2 N^{1/3} )$. We proceed by studying the conditional expectation in \eqref{eq:B_conditional_expectation}. To this end, for any $w \in L$, we denote $w'$ to be the last vertex in $H'$ on the unique path starting from $v_0$ and ending at $w$; such a vertex exists since $H$ is connected. We have that
\begin{align}
& \E \left[ \left. \frac{1}{N} \sum_{w \in L} \prod_{i = 1}^K e^{T_i(w) - T_i(v)} \mathbf{1} ( T_i(v) \ge T_i(w)) \right \vert \{ \mathbf{T}(z) \}_{z \in H'}, \mathbf{T}(v) \right] \nonumber \\
 & \hspace{2cm} = \frac{1}{N} \sum_{w \in L} \E \left[ \left. \prod_{i = 1}^K e^{T_i(w) - T_i(v) } \mathbf{1}( T_i(v) \ge T_i(w)) \right \vert \{ \mathbf{T}(z) \}_{z \in H'}, \mathbf{T}(v) \right] \nonumber \\
 & \hspace{2cm} = \frac{1}{N} \sum_{w \in L} \E \left[ \left. \prod_{i = 1}^K e^{T_i(w) - T_i(v)} \mathbf{1}(T_i(v) \ge T_i(w)) \right \vert \mathbf{T}(w'), \mathbf{T}(v) \right]  \nonumber \\
 \label{eq:B_conditional_expectation_terms}
 & \hspace{2cm} = \frac{1}{N} \sum_{w \in L} \prod_{i = 1}^K \E \left [ \left. e^{T_i(w) - T_i(v) } \mathbf{1}(T_i(v) \ge T_i(w)) \right \vert T_i(w'), T_i(v) \right]. 
\end{align}
Above, the second equality uses that the conditional expectation depends only on the random variables $T_i(w), T_i(v)$, and that the distribution of $T_i(w)$ conditioned on $\{ \mathbf{T}(z) \}_{z \in H'}$ is the same as the distribution conditioned on $T(w')$ in light of \eqref{eq:T_X_distribution}. The final equality follows from the independence of the $T_i(w)$'s over $i \in [K]$. 

We continue by lower bounding the conditional expectations in \eqref{eq:B_conditional_expectation_terms}. Recall that for $w \in L$, we have from \eqref{eq:T_X_distribution} that $T_i(w) = T_i(w') + X_i(w)$ where $X_i(w) \sim \mathrm{Gamma}(R, 1)$. Then
\begin{align}
& \E \left[ \left. e^{T_i(w) - T_i(v) } \mathbf{1} ( T_i(v) \ge T_i(w) ) \right \vert  T_i(w'), T_i(v) \right] \nonumber \\
& \hspace{2cm} = e^{T_i(w') - T_i(v) } \E \left[ \left. e^{X_i(w)} \mathbf{1} (  T_i(v) - T_i(w') \ge X_i(w) ) \right \vert  T_i(w'), T_i(v) \right] \nonumber \\
& \hspace{2cm} \stackrel{(a)}{=} \frac{e^{T_i(w') - T_i(v) }}{\Gamma(R) } \int_0^{T_i(v) - T_i(w') } x^{R - 1} dx\nonumber  \\
& \hspace{2cm}= \frac{ e^{T_i(w') - T_i(v) } ( T_i(v) - T_i(w') )^R }{\Gamma(R + 1) } \nonumber \\
& \hspace{2cm} \stackrel{(b)}{\ge} \frac{ e^{T_i(w') - T_i(v) } ( T_i(v) - T_i(w') )^R }{3 \sqrt{R} (R / e)^R } \nonumber \\
& \hspace{2cm}= \frac{1}{3 \sqrt{R}} e^{ - (T_i(v) - T_i(w') - R)} \left( 1 + \frac{T_i(v) - T_i(w') - R}{R} \right)^R \nonumber \\
& \hspace{2cm} \stackrel{(c)}{\ge}  \frac{1}{3 \sqrt{R}} \exp \left \{ - \frac{2 ( T_i(v) - T_i(w') - R)^2 }{R} \right \} \nonumber \\
& \hspace{2cm} \stackrel{(d)}{\ge} \frac{1}{3 \sqrt{R} } \exp \left \{ - \frac{4}{R} ( X_i(v) - R)^2 - \frac{4}{R} \left( T_i(v') + T_i(w') \right)^2 \right \} \nonumber \\
\label{eq:B_term_expectation_lower_bound}
& \hspace{2cm} \stackrel{(e)}{\ge} \frac{e^{-16}}{3 \sqrt{R}} \exp \left \{ - \frac{4}{R} ( X_i(v) - R)^2  \right \}.
\end{align}
In the display above, the equality $(a)$ uses that $T_i(v) - T_i(w') \ge 0$, which holds since $T_i(v) \ge 2R / 3 \ge \sqrt{R} \ge \max_{z \in G'} T_i(z)$ on the event $\cE_v$. The inequality $(b)$ is due to Stirling's approximation. The inequality $(c)$ follows from Lemma \ref{lemma:exponential_inequality}, which can be applied here since $T_i(v) - T_i(w') \ge R / 2$ on $\cE_v$. The inequality $(d)$ uses $(a + b)^2 \le 2a^2 + 2b^2$, and the inequality $(e)$ follows since $T_i(v'), T_i(w') \le \sqrt{R}$ on the event $\cE_v$. 

Substituting the lower bound \eqref{eq:B_term_expectation_lower_bound} into \eqref{eq:B_conditional_expectation_terms}, we obtain 
\begin{multline}
\label{eq:B_expectation_lower_bound}
 \E \left[ \left. \frac{1}{N} \sum_{w \in L} \prod_{i = 1}^K e^{T_i(w) - T_i(v)} \mathbf{1} ( T_i(v) \ge T_i(w)) \right \vert \{ \mathbf{T}(z) \}_{z \in H'}, \mathbf{T}(v) \right] \\
 \ge \left( \frac{e^{-16}}{3 \sqrt{R}} \right)^K \exp \left \{ - \frac{4}{R} \sum_{i = 1}^K ( X_i(v) - R)^2 \right \} \stackrel{(f)}{\ge} \left( \frac{e^{-30} }{\sqrt{R}} \right)^K,
\end{multline}
where $(f)$ holds on the event $\cE_v$. 

Together, \eqref{eq:B_conditional_expectation} and \eqref{eq:B_expectation_lower_bound} imply that conditioned on $\{ \mathbf{T}(z) \}_{z \in H'}, \mathbf{T}(v)$, and provided $\cE_v$ holds, we have with probability at least $1 - 2 \exp ( 2 N^{1/3} )$ that 
\[
\frac{1}{N} \sum_{w \in L} \prod_{i = 1}^K e^{T_i(w) - T_i(v)} \mathbf{1}( T_i(v) \ge T_i(w) ) \ge \left( \frac{e^{-30}}{\sqrt{R}} \right)^K - N^{-1/3} \ge R^{-K}.
\]
Finally, the desired result holds since $\p (\cE_v) = 1 - o(1)$.
\end{proof}

\section{Proof of Proposition \ref{prop:fpp}}
\label{sec:fpp}

It suffices to show that the model described in Proposition \ref{prop:fpp} satisfies the equation \eqref{eq:SI}. As \eqref{eq:SI} describes a \emph{unique} Markov process, the claim will then follow. We start by showing that it is unlikely for two vertices to become infected in a very small time interval.

\begin{lemma}
\label{lemma:cascade_multiple_agents}
For any $t \ge 0$, it holds that
\[
\p \left( \left. | \cascade [t, t + \epsilon] | \ge 2 \right \vert \cF_t \right) = o ( \epsilon).
\]
\end{lemma}

\begin{proof}
Define $\tau_1$ and $\tau_2$ to be the first and second times after $t$ when a new vertex is added to the cascade, respectively. 
Conditionally on $\cF_t$, the memoryless property of Exponential distributions implies that $\tau_1 - t$ is equal in distribution to $\mathrm{Exp}(\lambda |\cut( \cascade [0,t])| )$, which stochastically dominates a $\mathrm{Exp}( \lambda n^2 )$ random variable (here, we have used that the number of edges in $G$ is at most $n^2$). By the memoryless property of Exponential distributions, $\tau_2 - \tau_1$ is independent of $\tau_1 - t$. Through similar argument as above, it holds that $\tau_2 - \tau_1$ stochastically dominates a $\mathrm{Exp}( \lambda n^2)$ random variable as well. Consequently, the probability that at least two vertices join the cascade in the time interval $[t, t + \epsilon]$ is  
\begin{align*}
\p ( \tau_2 - t \le \epsilon \vert \cF_t ) & \le \p \left( \tau_1 - t \le \epsilon \text{ and } \tau_2 - \tau_1 \le \epsilon  \vert \cF_t \right) \\
& = \p \left( \tau_1 - t \le \epsilon \vert \cF_t \right) \p \left( \tau_2 - \tau_1 \le \epsilon \vert \cF_t \right) \\
& \le \left( 1 - e^{- \epsilon \lambda n^2 } \right)^2 \\
& \le \epsilon^2 \lambda^2 n^4 = o(\epsilon).
\end{align*}
\end{proof}

We now turn to the proof of the main result of this section. 

\begin{proof}[Proof of Proposition \ref{prop:fpp}]
We start by defining some relevant notation. Condition on $\cF_t$, and let $v \in V \setminus \cascade [0,t]$. Recalling that $\{ F(e) \}_{e \in E}$ is a collection of i.i.d. $\mathrm{Exp}(\lambda)$ random variables, we define
\begin{align*}
W & : = \min \{ F(w,v) : w \in \cascade[0,t], (w,v) \in E \} .
\end{align*}
By basic properties of Exponential distributions, we have that $W \sim \mathrm{Exp} ( \lambda | \cN(v) \cap \cascade [0,t] |)$.
Additionally define the event $\cE$ to hold if and only if at most one vertex joins the cascade in the interval $(t, t + \epsilon]$.

We claim that if $\cE$ holds, then $v \in \cascade [0, t + \epsilon] \iff W \le \epsilon$. Indeed, if $W \le \epsilon$, then from the description of $T(v)$ in the proposition statement, we have that $T(v) \le t + \epsilon$, which implies that $v \in \cascade [t, t + \epsilon]$. On the other hand, if $W > \epsilon$, then we consider two cases: (1) no vertex becomes infected in $[t, t + \epsilon]$ or (2) one vertex becomes infected in $[t, t + \epsilon]$. Note that no other case is possible if $\cE$ holds. In the first case, it holds for all $e \in \cut(\cascade [0,t])$ that $F(e) > \epsilon$, hence the weight of any path starting in $\cascade [0,t]$ and ending at $v$ is at least $\epsilon$, implying that $T(v) > t + \epsilon$. In the second case, if $v \in \cascade [t, t + \epsilon]$ despite $W > \epsilon$, there must exist a path $P \in \cP(v,t)$ of length at least 2 such that $\weight(P,t) \le \epsilon$. However, this implies the existence of another vertex $u$ on the path which is \emph{not} an element of $I[0,t]$. Since the portion of the path halting at $u$ has a smaller weight than $P$, it follows that $T(u) \in \cascade [t, t + \epsilon]$. However, it is not possible for both $u$ and $v$ to be elements of $\cascade [t, t + \epsilon]$ if $\cE$ holds, hence we must have that $T(v) > t + \epsilon$.

As a result of the argument above, we have that
\begin{align*}
\p \left( \{ v \in \cascade [0, t + \epsilon] \} \cap \cE \vert \cF_t \right) & = \p \left( \{ W \le \epsilon \} \cap \cE \vert \cF_t \right) \\
& = \p ( W \le \epsilon \vert \cF_t ) - \p \left( \{ W \le \epsilon \} \cap \cE^c \vert \cF_t \right) \\
& = \epsilon \lambda | \cN(v) \cap \cascade [0,t] | + o(\epsilon).
\end{align*}
Finally, since $\p ( \cE \vert \cF_t ) = 1 - o(\epsilon)$ by Lemma \ref{lemma:cascade_multiple_agents}, it holds that 
\[
\p \left( v \in \cascade [0, t + \epsilon] \vert \cF_t \right) = \epsilon \lambda | \cN(v) \cap \cascade [0,t] | + o(\epsilon),
\]
hence \eqref{eq:SI} is satisfied.
\end{proof}


{\small
\bibliographystyle{abbrv}
\bibliography{references}
}

\end{document}